\documentclass[12pt,reqno]{amsart} 
\usepackage{amsfonts, amssymb,stmaryrd}
\usepackage{latexsym}
\usepackage{graphicx}
\usepackage{xfrac} 
\usepackage{MnSymbol}
\usepackage{xcolor}
 \usepackage{hyperref}

\newtheorem{theorem}{Theorem}[section]
\newtheorem{corollary}[theorem]{Corollary}
\newtheorem{lemma}[theorem]{Lemma}
\newtheorem{proposition}[theorem]{Proposition}
\newtheorem{question}[theorem]{Question}

\theoremstyle{definition}
\newtheorem{definition}[theorem]{Definition}
\newtheorem{example}[theorem]{Example}
\theoremstyle{remark}
\newtheorem{remark}[theorem]{Remark}
\newtheorem{rem}[theorem]{Remark}

\newtheorem{clm}[theorem]{Claim}
\newtheorem{fact}[theorem]{Fact}

\numberwithin{equation}{section}

\renewcommand{\Im}{\mathrm{Im}}
\newcommand{\GL}{\mathrm{GL}}

\newcommand{\ppp}{\mathrm{PPP}}

\newcommand{\Cl}{\mathrm{Cl}}

\newcommand{\Mod}{\mathrm{Mod}}

\newcommand{\Stab}{\mathrm{Stab}}
\newcommand{\Homeo}{\mathrm{Homeo}}

\newcommand{\Core}{\mathrm{Core}}

\newcommand{\Span}{\mathrm{Span}}

\newcommand{\Fix}{\mathrm{Fix}}

\newcommand{\SL}{\mathrm{SL}}
\newcommand{\PSL}{\mathrm{PSL}}
\newcommand{\PGL}{\mathrm{PGL}}
\newcommand{\Sym}{\mathrm{Sym}}

\newcommand{\Aut}{\mathrm{Aut}}

\newcommand{\Inn}{\mathrm{Inn}}

\newcommand{\Out}{\mathrm{Out}}

\newcommand{\Top}[1]{\mathrm{\tau_{#1}}}

\newcommand{\arrow}{\rightarrow}
\newcommand{\action}{\curvearrowright}
\newcommand{\acts}{\curvearrowright}

\newcommand{\trivgp}{\langle e \rangle}

\newcommand{\R}{\mathbf R}

\newcommand{\Q}{\mathbf Q}
\newcommand{\N}{\mathbf N}
\newcommand{\Z}{\mathbf Z}
\newcommand{\F}{\mathbb F}

\newcommand{\Ac}{\mathcal{A}}
\newcommand{\Rc}{\mathcal{R}}
\newcommand{\Nc}{\mathcal{N}}

\newcommand{\Supp}{\operatorname{Supp}}
\newcommand{\Ncore}{\operatorname{Ncore}}
\newcommand{\PP}{\mathbb P}
\newcommand{\BP}{\mathbb P}
\newcommand{\MF}{\mathbb {M}}

\newcommand{\dist}{\mathrm{dist}}

\begin{document}
\title{Maximal subgroups of countable groups, a survey.}%
\author{Tsachik Gelander}
\author{Yair Glasner}
\author{Gregory So\u\i fer}
\maketitle




{\it To our friend, teacher and colleague Grisha Margulis with great admiration. While proving the most remarkable theorems in the field as well as many other fantastic results Margulis invented techniques and developed ideas that inspired so many mathematicians. This note describes one such idea and its many ramifications as it allowed us and others to solve problems and extend the theory.}


\begin{abstract}
This survey of the works of Margulis-So\u\i fer on maximal subgroups and of its many ramifications. 
\end{abstract}

\tableofcontents

\section{Introduction} \label{sec:intrro}

This paper is a survey on the works \cite{MS:ann,MS:first,MS:Maximal} on maximal subgroups in finitely generated linear groups, and the works that followed it  \cite{GG:Primitive,GG:AOS,GG:ht,Kapovich:Frattini,Ivanov:MCG,HO:ht,GM:max,AGS:gen,Soi:max1,Soi:max2,per:max,AKT:max,FG:max,GS:max,MF:LFsimple} concerning maximal subgroups of infinite index in linear groups as well as in various other groups possessing a suitable geometry or dynamics.

\subsection{The Margulis--So\u\i fer theorem}
The original motivation came from the following question of Platonov:
\begin{question}
Does $\SL_n(\Z), n \geq 3$ admit a maximal subgroup of infinite index?
\end{question}
In \cite{MS:ann,MS:first,MS:Maximal}  this question was answered positively. Moreover these papers clarified the existence question of infinite index maximal subgroups for all finitely generated linear groups:
\begin{theorem}\label{thm:MS} \cite{MS:ann,MS:first,MS:Maximal}
A finitely generated linear group admits a maximal subgroup of infinite index if and only if it is not virtually solvable.
\end{theorem}
 
The proof of theorem \ref{thm:MS} is inspired by Tits' proof of the classical Tits alternative \cite{Tits:alternative}. Recall that Tits proved that a finitely generated linear group $\Gamma$ which is not virtually solvable admits a free subgroup. In \ref{thm:MS} it is shown that in fact $\Gamma$ admits a profinitely dense free subgroup $F$. By Zorn's lemma $F$ is contained in a maximal proper subgroup $M$ of $\Gamma$. Since $F$ is profinitely dense, so is $M$ and therefore $[\Gamma:M]$ must be infinite. The details of the proof however are quite involved, especially in the case where $\Gamma$ is not Zariski connected. 

\subsection{Primitivity}
Every subgroup $H$ of a group $\Gamma$ corresponds to a transitive action of $\Gamma$, namely the action on the coset space $\Gamma/H$. The group $H$ is maximal if and only if this coset action is primitive in the sense of the following:
\begin{definition}
An action of a group $\Gamma$ on a set $X$ is {\it primitive} if $|X|>1$ and there are no
$\Gamma$-invariant equivalence relations on $X$ apart from the two trivial ones\footnote{The
trivial equivalence relations are those with a unique equivalence class, or with singletons as
equivalence classes. When $|X|=2$, one should also require that the action is not trivial.}. An
action is called {\it quasiprimitive} if every normal subgroup acts either trivially or
transitively. A group is {\it primitive} or {\it quasiprimitive} if it admits a faithful primitive
or quasiprimitive action on a set.
\end{definition}
In particular   $\Gamma$ has a maximal subgroup of infinite index if and only if it admits a primitive permutation action on an infinite set. Primitive actions form the basic building blocks of the theory of permutation groups. A lot of research was dedicated to the study of finite primitive groups (cf. \cite{AS:Maximal_subgroups, KL:survey_max, DM:Permutation_Groups}). The papers \cite{MS:ann,MS:first,MS:Maximal}  opened a door to the study permutation representation of infinite linear groups. 

The transition to permutation theoretic terminology suggests shifting the attention from infinite primitive groups to the study of groups admitting a faithful primitive action. This leads us to phrase the following guideline question. 
\begin{question} \label{q:main}
Characterize the countable primitive groups. 
\end{question}
Due to the method developed in \cite{MS:ann,MS:first,MS:Maximal} a satisfactory answer is within reach for many natural families of groups. Which brings us to the following definition, that will turn out to be central to our discussion:
\begin{definition} \label{def:ast}
A countable group $\Gamma$ is called {\it{of almost simple type}} if 
\begin{itemize}
\item it contains no nontrivial finite normal subgroups. 
\item if $M,N \lhd \Gamma$ with $[M,N]=\trivgp$ then either $M = \trivgp$ or $N = \trivgp$. In particular $\Gamma$ contains no nontrivial abelian normal subgroups. 
\end{itemize}
\end{definition}
As a direct consequence of Theorem \ref{thm:MS} one can prove the following:
\begin{theorem} 
An infinite finitely generated linear group $\Gamma$ is primitive if and only if it is of almost simple type.
\end{theorem}

The permutation representation viewpoint also suggests natural properties that are stronger than primitivity. 
\begin{definition} \label{def:2kh_tr}
An action $\Gamma \curvearrowright  X$ is called $2$-transitive if the induced action on pairs of distinct points $G\curvearrowright  X\times X\setminus \text{diag}(X)$ is transitive. An action $G\curvearrowright  X$ is $k$-transitive if it is transitive on $k$-tuples of distinct points, and is highly transitive if it is $k$-transitive for every $k$. A group will be called {\it{highly transitive}} if it admits a faithful, highly transitive action. 
\end{definition}
Every $2$-transitive action is primitive. Indeed if $\sim$ is a $\Gamma$ invariant equivalence relation on $X$ and if $x \sim y$ for some $x \ne y \in X$ then $2$-transitivity readily implies that any two points are equivalent. 
\subsection{The characterization of countable primitive linear groups}
Theorem \ref{thm:MS} was generalized to the setting of countable, but not necessarily finitely generated, linear groups. 
\begin{theorem}[\cite{GG:Primitive}]  \label{GMS}
Any countable linear nontorsion group of almost simple type is primitive. In fact such a group admits uncountably many non equivalent faithful primitive actions. 

Any countable linear nontorsion group which is not virtually solvable has, uncountably many, maximal subgroups of infinite index.   
\end{theorem}
In the zero characteristic case, as well as in the finitely generated case, the theorem remains valid without the assumption that the group is nontorsion. In positive characteristic, we need this assumption for our proof. In fact, as in the finitely generated case, the proof of Theorem \ref{GMS} actually establishes a stronger statement: the existence of a free subgroup which is contained in a maximal subgroup. This stronger statement fails for torsion groups like 
$\PSL_2(\overline{\F_7})$, where $\overline{\F_7}$ is the algebraic closure of the field $\F_7$, of seven elements. Note
however that $\PSL_2(\overline{\F_7})$ does not violate Theorem \ref{GMS} because it is primitive, and in fact even admits a faithful 3-transitive action on the projective line
$\mathbb{P}\overline{F_7}$.

Another difference that stands out between this theorem and its finitely generated counterpart, is the lack of the converse direction. The missing implication is actually the easy direction of Theorem \ref{thm:MS}. But, upon leaving the realm of finitely generated groups, it fails. An easy example is the $2$-transitive action of the solvable group $\Gamma = \left \{ \left. \begin{pmatrix} a & b \\ 0 & 1 \end{pmatrix} \ \right | \ a \in \Q^{*}, b \in \Q \right \}$ on the invariant set $\left \{ \left. \begin{pmatrix} x\\1\end{pmatrix} \in \Q^2 \ \right | \ x \in \Q \right \}$. This action can also be identified as the natural affine action of the semidirect product $\Q^{*} \ltimes \Q \action \Q$. 
\begin{definition} \label{def:banal}
Let $\Gamma = \Delta \ltimes M$ be a semidirect product. The natural affine (or standard) action of $\Gamma$ is the action $\Gamma \action M$ in which $M$ acts on itself by left translations and $\Delta$ acts on $M$ by the conjugation:
$$m\cdot x = mx, \ \forall m \in M, \qquad \delta \cdot x = \delta x \delta ^{-1}, \ \forall \delta \in \Delta.$$
\end{definition}

As it turns out this example is quite indicative as can be seen from the following two theorems.
\begin{theorem} \label{thm:non_ast1}
Let $\Gamma$ be a primitive countable group which is not of almost simple type. Then $\Gamma$ splits as a semidirect product $\Gamma = \Delta \ltimes M$ and the given primitive action is equivalent to its natural affine action on $M$. 
\end{theorem}
In particular it follows that the faithful primitive action is unique in this case. In fact it is even unique amongst all faithful quasiprimitve actions of this group. Of course the same group might admit additional primitive actions that are not faithful. Take for example the group $\SL_2(\Q) \ltimes \Q^2$. Its natural action on $\Q^2$ is 2-transitive and hence the unique primitive faithful action of this group by Theorem \ref{thm:non_ast1}. However this group admits a quotient which is of almost simple type $\SL_2(\Q)$. Thus by Theorem \ref{GMS} it does admit uncountably many primitive actions factoring through this quotient. 

The semidirect products whose natural action is primitive/faithful are easily classified
\begin{theorem} \label{thm:semidir1}
The affine action of a semidirect product $\Gamma = \Delta \ltimes M \action M$ is faithful iff $Z_{M}(\Delta) = \{\delta \in \Delta \ | \ [\delta,m]=e \ \forall m \in M\} = \trivgp$. This action is primitive if and only if the only subgroups of $M$ that are normalized by $\Delta$ are $M$ itself and the trivial group $\trivgp$. 
\end{theorem}
\begin{definition}
Let $\Gamma = \Delta \ltimes M$ be a countable, semidirect product whose natural action is primitive and faithful. That is such that $M$ is characteristically simple, and admits no nontrivial $\Delta$-invariant subgroups and $Z_{M}(\Delta) = \trivgp$. Then, if in addition $M$ is abelian $\Gamma$ is called {\it{primitive of affine type}} and if $M$ is nonabelian $\Gamma$ is called {\it{primitive of diagonal type}}. 
\end{definition}

Combining Theorem \ref{GMS} with the elementary classification of primitive groups that are not of almost simple type, yields a characterization of countable linear primitive groups. Subject to the additional assumption that the groups in question contain at least one element of infinite order. We like to think of this theorem as a rough generalization of the Aschbacher-O'Nan-Scott theorem (see \cite{AS:Maximal_subgroups,DM:Permutation_Groups}) to the setting of countable linear groups. 
\begin{theorem}\label{thm:Linear}
A countable nontorsion linear group $\Gamma$ is primitive if and only if one of the following
mutually exclusive conditions hold.
\begin{itemize}
\item \label{itm:ZClosure} $\Gamma$ is primitive of almost simple type.
\item \label{itm:affine} $\Gamma$  is primitive of affine type.
\item \label{itm:diagonal} $\Gamma$ is primitive of diagonal type.
\end{itemize}
In the affine and the diagonal cases the group admits a unique faithful quasiprimitive action. In the almost simple case the group admits uncountably many non-isomorphic faithful primitive actions.
\end{theorem}
\begin{rem}
 For a finitely generated group $\Gamma$ only the first possibility can occur in the above theorem. 
\end{rem}
\begin{rem}
 In all the cases under consideration, it is shown that a group is primitive if and only if it is quasiprimitive. 
\end{rem}
\begin{example} \label{eg:commensurable}
Let $\Sigma$ be any simple countable linear group which is not torsion. For example one can take $\Sigma = \PSL_2(\Q)$. Now consider the two groups $\Gamma_1 = \Sigma \times \Sigma$ and $\Gamma_2 = (\Z/2\Z) \ltimes (\Sigma \times \Sigma)$. Despite the clear similarity between these two groups (one being an index two subgroup in the other), their respective permutation representation theories are quite different. $\Gamma_1$ is primitive of diagonal type and hence it admits a unique faithful primitive action. Moreover this action is very explicit, it is the action $\Gamma_1 \action \Sigma$ given by $(\gamma_1, \gamma_2) \cdot \sigma = \gamma_1 \sigma \gamma_2^{-1}$. On the contrary the group $\Gamma_2$ is primitive of almost simple type and hence admits uncountably many, non-isomorphic faithful primitive actions. Yet, we do not have a good explicit description for any of these actions.
\end{example}
\begin{example} \label{eg:pslnq}
Let $\Gamma = \PSL_n(\Q)$. This group admits a very explicit faithful primitive action, namely its action on the projective line $\PP(\Q^n)$. When $n=2$ this action is not only primitive but it is also $3$-transitive. Being a group of almost simple type, the above theorem yields uncountably many other non-isomorphic primitive permutation representations. Again, we do not have any explicit descriptions of these actions. 
\end{example}

Section \ref{sec:nec} is dedicated to the classification of countable primitive groups that are not of almost simple type. It deals with general countable groups, and uses only soft, group theoretic arguments. Section \ref{sec:lin} is dedicated to the linear case and the proof of Theorem \ref{thm:Linear}.

\subsection{The variety of maximal subgroups}
Since the construction of maximal subgroups of infinite index in \cite{MS:ann,MS:first,MS:Maximal}, it is expected that there should be examples of such maximal subgroups of various different natures. In particular in the latter paper the existence of uncountably many maximal subgroups in any finitely generated non virtually solvable linear group was established. However, as the proof is non-constructive and relies on the axiom of choice, it is highly nontrivial to lay one's hands on specific properties of the resulting groups. 

In many special cases one can find examples of maximal subgroups in the same group that are very different from each other. Example \ref{eg:pslnq} describes actions of $\PSL_n(Q)$, some of which have Zariski dense stabilizers and others not. For $\SL_2(\Q)$ one even has a 3-transitive action with a solvable point stabilizer. In another direction, Section \ref{sec:ht} constructs highly transitive faithful actions for any nontorsion group of almost simple type $\Gamma < \SL_2(k)$, where $k$ is any local field. Many of these groups admit also actions which are not highly transitive. These examples and many more come to show that there is probably a whole zoo of maximal subgroups out there that we are only starting to see. 

To some extent, a benchmark example is the group $\SL_n(\Z), \ n \ge 3$. The same groups appearing in the original question of Platonov. These groups are very rigid in nature. And it is quite possible that a good understanding of the family of maximal subgroups here would shed light on the general case.

The first step would be to show that indeed, maximal subgroups $\Delta\le\SL_n(\Z)$ of different nature do exist. As of today, little can be said about the intrinsic algebraic structure of $\Delta$. Instead, one is lead to focus on the way it sits inside $\SL_n(\Z)$. Two point of views that are natural to consider are:

\begin{itemize}
\item The associated permutation representation $\Gamma \acts \Gamma/\Delta$.
\item The action of $\Delta$ on the associated projective space $\PP^{n-1}(\R)$. 
\end{itemize}

The following results were established in \cite{GM:max}: 
\begin{theorem}\label{thm counting} 
Let $n \ge 3$. There are $2^{\aleph_0}$ infinite index maximal subgroups in $\SL_n(\Z)$.
\end{theorem}

\begin{theorem}\label{thm dense} 
Let $n \ge 3$. There exists a maximal subgroup $\Delta$ of $\SL_n(\Z)$ which does not have a dense orbit in $\PP^{n-1}(\R)$.  
In particular, the limit set of $\Delta$ (in the sense of \cite{CG00}) is nowhere-dense. 
\end{theorem}

\begin{theorem}\label{thm trivial} 
Let $n \ge 3$. There exists an infinite index maximal subgroup $M$ of  $\PSL_n(\Z)$  and an element $g \in \PSL_n(\Z)$
such that $M \bigcap gMg^{-1}=\{id\}$.
\end{theorem}  

\begin{theorem}\label{thm not 2-trans} 
Let $n \ge 3$. There exists a primitive permutation action of $\SL_n(\Z)$ which is not 
2-transitive. 
\end{theorem} 

\begin{rem}\normalfont These theorems remain true also for 
$\SL_n(\Q)$ instead of $\SL_n(\Z)$.
\end{rem}

\begin{rem}\normalfont
Recall that Theorem \ref{thm:MS} is much more general. It holds for any finitely generated non-virtually-solvable linear group $\Gamma$. However the proof of the last results rely on special properties of $\SL_n(\Z),~n\ge 3$. In particular
one important ingredient here is the beautiful result of Venkataramana about commuting unipotents, Theorem \ref{Thm - Venkataramana}. Another ingredient is the result of Conze and Guivarc'h, Theorem \ref{them - conze}. Some of these results can be extended to the class of arithmetic groups of higher $\Q$-rank.
\end{rem}

\subsection{Highly transitive actions}
Over the years many authors generalized the results and the methods of  \cite{MS:ann,MS:first,MS:Maximal}. In Section \ref{sec:ht} we describe two major directions which eventually converged together in a very nice way. The first direction involves implementing the methods of \cite{MS:Maximal} to various {\it{linear-like}} setting. Notable examples include the works of I. Kapovich \cite{Kapovich:Frattini} for subgroups of hyperbolic groups and of Ivanov \cite{Ivanov:MCG} on mapping class groups. More specifically than just linear-like, these examples exhibit boundary dynamics that is closer in nature to that of subgroups of rank-1 Lie groups. From the group theoretic point of view this has the effect that often after ruling out obvious obstructions, all groups in question are of almost simple type and primitive. 

The second direction involves construction $k$-transitive or highly transitive permutation representations which are A-priori harder to construct. Originally the predominant feeling highly transitive groups are much rarer. However, over a period spanning a few decades, wider and more elaborate constructions of highly transitive groups were given. Some notable papers in this direction are \cite{MD_free_HT, Dixon:free_HT, hic:ht, gun:ht, Kit:ht, MS:ht, chay:ht, GG:ht, HO:ht}. The paper of Hull and Osin establishes the following: 
\begin{theorem}[\cite{HO:ht} Theorem 1.2] \label{thm:HO} Any countable acylindrically hyperbolic group with no finite normal subgroups is highly transitive 
\end{theorem}
The family of acylindrically hyperbolic groups is very wide, encompassing within it most groups that could be considered of as {\it{rank one}}, even in a weak sense of the word. Thus, while Theorem \ref{thm:HO} does not imply the Theorem \ref{thm:MS} it ties together the strings, establishing a very strong form of this theorem for a wide range of rank one examples, thereby generalizing the theorems mentioned above. To emphasise this we quote from their paper a few specific situations where their theorem applies. 
\begin{itemize}
\item A countable group relatively hyperbolic to a collection of proper subgroups is highly transitive iff it is not virtually cyclic and has no finite normal subgroups. 
\item $\Mod(\Sigma_{g,n,p})$ the mapping class group of a compact orientable surface, of genus $g$, $p$-punctures and $n$-boundary components, is highly transitive if and only if $n=0$ and $3g+p \ge 5$. 
\item $\Out(F_n)$ is highly transitive iff $n \ge 3$. 
\item $\pi_1(M)$ where $M$ is a compact irreducible 3-manifold is highly transitive iff it is not virtually solvable and $M$ is not Seifert fibered. 
\item A right angled Artin group is highly transitive if and only if it is non-cyclic and directly indecomposable. 
\end{itemize}
We can summarize all of the above by saying that in each and every one of these situations, a subgroup is highly transitive if and only if it is of almost simple type. 


Inspired by the work of Hull and Osin, and applying the theory of linear groups, we establish the following result.  

\begin{theorem} \label{thm:rank_one_ht} Let $k$ be a local field, and $\Gamma < \SL_2(k)$ a center free unbounded countable group. Then the following conditions are equivalent for $\Gamma$. 
\begin{enumerate}
\item \label{r1_ast} $\Gamma$ is of almost simple type
\item \label{r1_zd} $\Gamma$ is Zariski dense 
\item \label{r1:nvs} $\Gamma$ is not virtually solvable
\item \label{r1_ht}$\Gamma$ is highly transitive
\end{enumerate}
\end{theorem}

We note that this is the only result in this survey paper which has not been proved before.

From a group theoretic point of view we use Hull-Osin's characterization of the point stabilizers in highly transitive actions. From the geometric point of view, we use the topological dynamics of the boundary action, instead of the small cancelation type methods used in \cite{HO:ht}. 

\subsection{Other geometric settings}
As seen so far the ideas behind the proof of Theorem \ref{thm:MS} have been substantially generalized to many ``linear-like'' settings, by which we refer very loosely to the situation where there is a group action with a rich enough proximal dynamics. The outcome of the theory in many of these settings is that all groups of almost simple type are primitive. Sometimes, notably in negatively curved type settings much stronger transitivity properties are established by similar techniques. 

In Section \ref{sec:nonlin} we offer just a glimpse, without proofs, into some fascinating works outside the linear-like setting. Here we encounter completely different methods and different types of behaviour. 

The research of the first two authors was partially funded by Israel Science Foundation grant ISF 2919/19. 

\section{Faithful primitive actions of countable groups} \label{sec:nec}

\subsection{Necessary conditions and group topologies on \(\Gamma\)} \label{sec:necessary}
The goal of this section is to establish Theorems \ref{thm:non_ast1} and \ref{thm:semidir1} from the introduction. A more detailed version of these, that appears here as Proposition \ref{prop:nec}, can be thought of as a summary of those implications in Theorem \ref{thm:Linear} which hold for general countable groups without any additional assumptions. It provides the classification of countable primitive groups into three disjoint classes: primitive groups of affine, diagonal and of almost simple types. The structure of primitive groups of affine or diagonal type is well understood. But the structure of primitive groups of almost simple type remains mysterious, this is the class of groups for which Question \ref{q:main} is the most interesting. We start with the proof of Theorem \ref{thm:semidir1}, characterizing primitivity and faithfulness for the standard actions of semidirect products. 
\begin{proof}[Proof of Theorem \ref{thm:semidir1}]
Consider the natural action of a semidirect product $\Gamma = \Delta \ltimes M \action M$. One easily verifies that $\Delta = \Stab_{\Gamma}(\{e\})$ and hence the kernel of the action is $Z_M(\Delta) = \{\delta \in \Delta \ | \ \delta m \delta^{-1}=m, \ \forall m \in M\}$. Which proves the first statement. If $\trivgp \neq K \lneqq M$ is normalized by $\Delta$ then $\Delta < \Delta \ltimes K < \Gamma$ so $\Delta$ is not maximal and the standard action fails to be primitive. Conversely assume that $\Delta$ fails to be a maximal with $\Delta \lneqq \Sigma \lneqq \Gamma$ an intermediate subgroup. Using the unique product decomposition $\Gamma = \Delta M$, we find that $\trivgp \ne \Sigma \bigcap M$  is a nontrivial proper subgroup of $M$ normalized by $\Delta$. 
\end{proof}
This gives rise to a very explicit description of primitive groups of affine and of diagonal type.
It follows that in a primitive group of affine or diagonal type the normal subgroup $M$ must be characteristically simple. When $M$ is abelian and countable this means that $M$ is the additive group of a countable vector space, i.e. either $M \cong \F_p^{\infty}$ or $M \cong \Q^n, n \in \N \bigcup \{\infty\}$. In this case $\Delta$ can be identified with an irreducible subgroup of $\GL(M)$. 

When $M$ is nonabelian it is center free, so the natural map $\iota: M \rightarrow \Inn(M) < \Aut(M)$ is injective. Since there are no nontrivial $\Delta$-invariant subgroups $\Delta \bigcap \Inn(M)$ is either trivial or equal to $\Inn(M)$. In the first possibility $\Delta$ commutes with $\Inn(M)$ contradicting the fact that there are no $\Delta$-invariant subgroups in $M$; hence $\Inn(M) < \Delta$. Recall that we identified the action of $M$ with its left action on itself. Now that we have realized the inner automorphisms of $M$, as a subgroup of $\Delta$ we can distinguish another subgroup of $\Gamma$, which is isomorphic to $M$ and commutes with it, namely the action of $M$ on itself from the right: 
$$N = \{\iota(m^{-1}) m \ | \ m \in M\} < \Gamma.$$
This is of course the source of the name ``diagonal type''. Indeed we have constructed here the diagonal action $M \times M \action M$ described in Example \ref{eg:commensurable} as a subaction of any primitive group of diagonal type. We turn to the proof of Proposition \ref{prop:nec} which is a more detailed version of Theorem \ref{thm:non_ast1} from the introduction. 
\begin{proposition} \label{prop:nec}
Let $\Gamma \action \Omega$ be a faithful primitive action of a countably infinite group on a set. Fix a basepoint $\omega_0 \in \Omega$ and let $\Delta = \Gamma_{\omega_0}$ be its stabilizer. Then every nontrivial normal subgroup of $\Gamma$ is infinite. Moreover $\Gamma$ falls into precisely one of the following three categories:
\begin{enumerate} 
\item \label{itm:is_ast} Either $\Gamma$ is primitive of almost simple type.  
\item \label{itm:affine} Or $\Gamma$ is primitive of affine type. In this case $\Gamma = \Delta \ltimes V$ with $V$ a countable vector space over a prime field $F$ (possibly $F=\Q$) and $\Delta$ an irreducible subgroup of $\GL_F(V)$. The given action is equivalent to the standard affine action of $\Gamma$ on $V$.  
\item \label{itm:diag} Or $\Gamma$ is primitive of diagonal type. In this case $\Gamma = \Delta \ltimes M$ with $M$ a nonabelian, characteristically simple group. $\Inn(M) < \Delta < \Aut(M)$ and $M$ has no nontrivial subgroups that are invariant under the $\Delta$-action. Again the given action in this case is equivalent to the standard action of $\Gamma$ on $M$. 
\end{enumerate}

Conversely if $\Gamma$ falls into categories (\ref{itm:affine}) or (\ref{itm:diag}) then its natural action is faithful and primitive. 
\end{proposition}
\begin{proof}
As the action is faithful, $\Omega$ must be infinite so it is absurd that a finite normal subgroup would act transitively. If $\Gamma$ fails to be of almost simple type then we can find $\trivgp \ne M, N \lhd \Gamma$ with $[M,N]=\trivgp$. Both normal subgroups must act transitively on $\Omega$ since the action is primitive and faithful. But $M_{n \omega_0} = n M_{\omega_0} n^{-1} = M_{\omega_0}, \ \forall n \in N$ so that the stabilizer $M_{\omega_0}$ will fix all of $N \omega_0 = \Omega$ which implies $M_{\omega_0} = \trivgp$; by the faithfulness of the action. Thus $M$ acts regularly on $\Omega$. Identifying $\Omega$ with $M$ via the orbit map $m \mapsto m \cdot \omega_0$ it is now routine to verify that $\Gamma = \Delta \ltimes M$ and that the given action of $\Gamma$ on $\Omega$ is equivalent to the standard affine action of this semidirect product. The question whether the group is of affine or diagonal type now depends only on the whether $M$ is abelian or not. The detailed description of the structure of the group in these two cases follows from our discussion above.

Let us just note that, just like $M$, the group $K = M \bigcap N$ is a normal subgroup commuting with $N$. So, either $K = \trivgp$ or $K$ too acts regularly, by the same argument in which case $M = N$. These two cases of course correspond to the diagonal and the affine cases respectively.  
\end{proof}

This theorem highlights the family of {\it{almost simple groups}} as the family for which Question  \ref{q:main} is interesting. In all other cases the existence question of a faithful primitive action is easily resolved. Proposition \ref{prop:nec} is very similar in its structure to our main Theorem \ref{thm:Linear}. It basically summarizes all the ``easy'' implications of that theorem that do not require the linearity assumption. The remaining implication does not work without the linearity as the following example shows. Section \ref{sec:nonlin} will be dedicated to a more comprehensive treatment of such examples. 
\begin{theorem}[\cite{per:max}]
The first Grigorchuk group is an example of a group that is of almost simple type but does not admit any infinite primitive action. 
\end{theorem}

\subsection{The normal topology}
We like to think of the definition of groups of almost simple type in topological terms. Let 
$$\Nc(\Gamma) = \left \{\trivgp \ne N \lhd \Gamma \right \}$$ be the collection of all nontrivial normal subgroups. If $\Gamma$ is of almost simple type every $N \in \Nc(\Gamma)$ is infinite and this family is closed under intersections. Clearly it is invariant under conjugation. Thus it forms a basis of identity neighborhoods for a group topology on $\Gamma$.
\begin{definition}
The {\it{normal topology}} $\tau_{N}$, on a countable group of almost simple type $\Gamma$ is the topology obtained by taking $\Nc(\Gamma)$ as a basis of open neighborhoods of the identity. 
\end{definition} 
\begin{proposition}
The normal topology is second countable. It is finer\footnote{Possibly the two topologies are equal, this happens exactly when the group admits the Margulis normal subgroup property.} than the profinite topology. It is Hausdorff if and only if $\Ncore(\Gamma) = \bigcap_{N \in \Nc(\Gamma)} N = \trivgp$.
\end{proposition}
\begin{proof}
We denote by \(\llangle g \rrangle = \langle \gamma g \gamma^{-1} \ | \ \gamma \in \Gamma \rangle \lhd \Gamma\) the normal subgroup of \(\Gamma\) generated by the conjugacy class of the element \(g \in \Gamma\). The countable collection \(\{\llangle g \rrangle \ | \ g \in G\}\) forms a basis of identity neighborhoods for the normal topology. Thus the topological group \((\Gamma, \Top{N})\) is first (and consequently also second) countable. All the rest of the statements are obvious. 
\end{proof}

In \cite[Definition 2.3]{GG:AOS} we gave a different definition for linear groups of almost simple type. The following lemma shows that in the specific case of countable linear groups these two definitions agree.
\begin{lemma} \label{lem:linear_ast}
Let \(\Gamma\) be a countable linear group. Then \(\Gamma\) is of almost simple type if and only if there exits a faithful linear representation \(\phi: \Gamma \arrow \GL_n(k)\), with \(k\) algebraically closed. Such that the Zariski closure \(G = \left(\overline{\phi(\Gamma)}^Z \right)\) satisfies $G^{(0)} = H^m$ where $H$ is a simple, center free, algebraic group and \(\Gamma\) acts transitively by conjugation on these \(m\) simple factors. Moreover for any such representation if $\trivgp \ne N \lhd \Gamma$ then $G^{(0)} < \left(\overline{\phi(N)}^Z \right)$. 
\end{lemma} 
\begin{proof} Assume that \(\Gamma\) is linear of almost simple type. Realize \(\Gamma < \GL_n(k)\) as a linear group, over an algebraically closed field \(k\) and let \(G = \overline{\Gamma}^{Z}\) be  the Zariski closure. We may assume, without loss of generality, that $G$ is semisimple. If not, we divide out by the solvable radical of $G$. Since \(\Gamma\) has no nontrivial abelian normal subgroups it also has no nontrivial solvable normal subgroups and it will still map injectively into the semisimple quotient. Next we may assume without loss of generality that  $G$ is adjoint upon replacing it by its image under the adjoint representation.  Since \(\Gamma\) has no finite or abelian normal subgroups it maps injectively into this new group too. Note that now $G^{(0)} = H_1 \times H_2 \times \ldots \times H_l$ is a direct product of simple, center free groups. $\Gamma$ acts on $G^{(0)}$ by conjugation. We claim that this action must permute these $l$ simple factors. Fix $\gamma \in \Gamma$ and $1 \le i,j \le l$, by simplicity of these factors there are only two options either $\gamma H_i \gamma^{-1} = H_j$ or $\gamma H_i \gamma^{-1} \bigcap H_j = \trivgp$. When the second of the two options holds, we have $[\gamma H_i \gamma^{-1}, H_j]=e$. Since $\gamma H_i \gamma^{-1}$ cannot commute with the whole group, there must be a (necessarily unique) $j$ such that $\gamma H_i \gamma^{-1} = H_j$. Let us denote the permutation representation thus obtained by $\gamma H_i \gamma^{-1} = H_{\pi(\gamma,i)}$. After rearranging we can rewrite the decomposition of the connected component into simple factors as follows $G^{(0)} = G_1 \times \ldots \times G_k$ where each $G_i$ is the direct power of simple factors $G_i = H_i^{m_i}$ that are permuted transitively by $\Gamma$. This gives rise to an injective map $\phi: \Gamma \rightarrow \Aut(G_1) \times \ldots \times \Aut(G_k)$ and since we assumed that $\Gamma$ is subdirect irreducible we can find a factor $i_0$ such that the map $\phi_{i_0}: \Gamma \rightarrow \Aut(G_{i_0})$ is already injective. This is our desired quotient. 

Conversely assume that $\Gamma$ admits such a linear representation. If $\trivgp \ne N \lhd \Gamma$ then $W=\left(\overline{\phi(N)}^Z \right) \lhd G$. If $N$ were either finite, or abelian, $W$ would be the same, contradicting the fact that $G$ does not admit finite or abelian normal subgroups. 

By assumption $N$ is infinite so that $W$ has positive dimension and $W \bigcap G^{(0)} \neq \trivgp$. Which means that $W \bigcap H \ne \trivgp$ for at least one of the simple factors $G^{(0)} = H^m$. By simplicity $W$ actually contains this simple factor, and by transitivity of the $\Gamma$ action on the factors $W>G^{(0)}$. This establishes the last statement. 

Finally assume that  $N_1 \bigcap N_2 = \trivgp$ for two normal subgroups $\trivgp \ne N_1, N_2 \lhd \Gamma$. Then $[N_1, N_2]=\trivgp$. Passing to the Zariski closure $W_i = \overline{N_i}^Z$, we have $[W_1,W_2]=\trivgp$. But by the previous paragraph this implies that $G^{(0)}$ is abelian which is absurd. This concludes the proof of the lemma.
\end{proof}
\begin{definition}
Let $\Gamma$ be a group of almost simple type. A subgroup $\Delta < \Gamma$ is called {\it{prodense}} if it is dense in the normal topology. 
\end{definition}
\begin{definition}
The core of a subgroup $\Delta < \Gamma$ is the normal subgroup $\Core_{\Gamma}(\Delta) = \bigcap_{\gamma \in \Gamma} \gamma \Delta \gamma^{-1}$. The subgroup $\Delta$ is called {\it{core free}} if $\Core_{\Gamma}(\Delta) = \trivgp$. 
\end{definition} 
Note that $\Core_{\Gamma}(\Delta)$ is exactly the kernel of the permutation representation $\Gamma \curvearrowright \Gamma/\Delta$. Consequently this action is faithful if and only if $\Delta$ is core free. We leave the verification of the following easy lemma to the reader. 
\begin{lemma}
The following conditions are equivalent for a subgroup $\Delta$ in a group $\Gamma$ of almost simple type:
\begin{itemize}
\item $\Delta$ is prodense
\item $\Delta N = \Gamma$ for every $\trivgp \ne N \lhd \Gamma$. 
\item $\phi(\Delta) = \phi(\Gamma)$ whenever $\phi:\Gamma \arrow H$ is a homomorphism with nontrivial kernel. 
\end{itemize}
\end{lemma}
\noindent 
\noindent In particular prodense subgroups are always core free. Clearly every prodense subgroup is profinitely dense. 

Such subgroups play a central role in the strategy for construction of maximal subgroups. The idea is that if $\Delta < \Gamma$ is a prodense (resp. profinitely dense subgroup) then any larger subgroup still has the same property. In particular if $\Delta < M < \Gamma$ is a proper maximal subgroup containing $\Delta$ then $M$ too must be prodense (resp. profinitely dense). This in turn guarantees that $M$ is core free (resp. of infinite index). In order to make sure that $\Delta$ is contained in some proper maximal subgroup we use the following condition. 
\begin{definition} \label{def:cfg}
A subgroup $\Delta < \Gamma$ is called {\it{cofinitely generated}} if $\Gamma = \langle \Delta, F\rangle$ for some finite $F \subset \Gamma$.
\end{definition}
If $\Delta < \Gamma$ is cofinitely generated let $F$ be a finite subset as in the above definition. Zorn's lemma yields a subgroup $\Delta < M < \Gamma$ which is maximal with respect to the condition that $F \not \subset M$. Any strictly larger subgroup will contain $\Gamma = \langle M , F \rangle$, so $M$ must be a maximal subgroup. To summarize the discussion:
\begin{proposition} \label{prop:cfgpd}
Let $\Gamma$ be any group. 
\begin{enumerate}
\item $\Gamma$ contains a maximal subgroup of infinite index if and only if it contains a profinitely dense cofinitely generated subgroup. 
\item  $\Gamma$ contains a core free maximal subgroup if and only if it contains a prodense cofinitely generated subgroup. 
\end{enumerate}
\end{proposition}
\section{Linear groups} \label{sec:lin}
For countable linear groups, there is almost a complete answer to Question \ref{q:main}. The missing part of the puzzle is the case of amenable countable linear groups of almost simple type. We give a more or less complete description of the what is known. Starting with some preliminaries. 
\subsection{Projective transformations over valuation fields} \label{sec:pt}
In this paragraph we shall review some definitions and results from \cite{BG:Dense_Free} and
\cite{BG:Topological_Tits} regarding the dynamical properties of projective transformations which
we shall use in the proof.

Let $k$ be a local field and $\left\| \cdot\right\| $ the standard norm on $k^{n}$, i.e. the
standard Euclidean norm if $k$ is Archimedean and $\left\| x\right\| =\max_{1\leq i\leq n}|x_{i}|$
where $x=\sum x_{i}e_{i}$ when $k$ is non-Archimedean and $(e_{1},\ldots ,e_{n})$ is the canonical basis of $k^{n}$. This norm extends in the usual way to $\Lambda ^{2}k^{n}$. We define the \textit{standard metric} on $\PP(k^{n})$ by $d(\overline{v},\overline{v})=\frac{\left\| v \wedge w\right\|}{\left\| v\right\| \left\| w\right\|}$, where $\overline{v}$ denotes the projective point corresponding to $v\in k^n$. 
$\left\| v \wedge w\right\|$ can be thought of as the `area' of the parallelogram defined by the vectors $v,w$, i.e. $\left\| v\right\|$ times the distance of $w$ to the line spanned by $v$.
Unless otherwise specified all our notation will refer to this metric. In particular $(A)_{\epsilon} :=\{x \in \PP(k^n) \ | \ d(x,A) < \epsilon\}$ will denote the open $\epsilon$-neighborhood and $[A]_\epsilon:=\{x \in \PP(k^{n}) \mid d(x,A) \le \epsilon\}$ the closed neighborhood of a set $A \subset \PP(k^n)$. With respect to this metric, every projective transformation is bi-Lipschitz on $\mathbb{P}(k^{n})$. 
\begin{definition} \label{def:contraction}
For $\epsilon \in (0,1)$, we call a projective transformation $g \in \PGL_{n}(k)$ {\it{$\epsilon-$contracting}} if $g \left( \PP(k^n) \setminus (\overline{H})_{\epsilon} \right) \subset (\overline{v})_{\epsilon}$ for some point $\overline{v} \in \PP(k^n)$ and projective hyperplane $\overline{H} < \PP(K^n)$, which are referred to as an attracting point and a repelling hyperplane for $g$. We say that $g$ is {\it{$\epsilon-$very contracting}} if $g$ and $g^{-1}$ are $\epsilon $-contracting. A projective transformation $g \in \PGL_{n}(k)$ is called {\it $(r,\epsilon )$-proximal} ($r>2\epsilon >0$) if it is $\epsilon$-contracting with respect to some attracting point $\overline{v} \in \mathbb{P}(k^{n})$ and repelling hyperplane $\overline{H}$, such that $d(\overline{v},\overline{H})\geq r$. A projective transformation $g$ is called {\it $(r,\epsilon )$-very proximal} if both $g$ and $g^{-1}$ are $(r,\epsilon )$-proximal. Finally, $g$ is simply called {\it proximal} (resp. {\it very proximal}) if it is $(r,\epsilon )$-proximal (resp. $(r,\epsilon )$-very proximal) for some $r>2 \epsilon >0$.
\end{definition}

The attracting point $\overline{v}$ and repelling hyperplane $\overline{H}$ of an $\epsilon $-contracting transformation are not uniquely defined. Yet, if $g$ is proximal with good enough parameters we have the following natural choice for an attracting point and a repelling hyperplane:
\begin{lemma}\label{fix}(Lemma 3.1 of \cite{BG:Topological_Tits})
Let $k$ be a local field and $\epsilon \in (0,\frac{1}{4})$. There exist two constants $c_{1},c_{2}\geq 1$ such that if $g \in \PGL_n(k)$ is an $(r,\epsilon)$-proximal transformation with $r \geq c_{1} \epsilon$ and associated attracting point $\overline{v}$ and repelling hyperplane $\overline{H}$. Then $g$ fixes a unique point $\overline{v}_{g} \in \left(\overline{v} \right)_{\epsilon}$ and a unique projective hyperplane $\overline{H}_{g} \subset \left(\overline{H}\right)_{\epsilon}$. Moreover, if $r \geq c_{1}\epsilon ^{2/3}$, then the positive powers $g^{n}$, $n\geq 1$, are $(r-2\epsilon ,(c_{2}\epsilon)^{\frac{n}{3}})$-proximal transformations with respect to these same $\overline{v}_{g}$ and $\overline{H}_{g}$. The constants $c_1,c_2$ may depend on the local field $k$, but they become only better when passing to a finite extension field. 
\end{lemma}

\begin{remark}
In what follows, whenever we add the article \textit{the} (or {\it the canonical}) to an attracting point and repelling hyperplane of a proximal transformation $g$, we shall mean these fixed point $\overline{v}_{g}$ and fixed hyperplane $\overline{H}_{g}$ obtained in Lemma \ref{fix}.  

Moreover, when $r$ and $\epsilon$ are given, we shall denote by
$A(g),R(g)$ the $\epsilon$-neighborhoods of $\overline{v}_g,\overline{H}_g$ respectively. In some cases, we shall specify different attracting and repelling sets for a proximal element $g$. In such a case we shall denote them by $\mathcal{A}(g),\mathcal{R}(g)$ respectively. This means that
$$g \big(\PP(k^n)\setminus\mathcal{R}(g)\big)\subset\mathcal{A}(g). $$
If $g$ is very proximal and we say that $\mathcal{A}(g),\mathcal{R}(g),\mathcal{A}(g^{-1}),\mathcal{R}(g^{-1})$ are specified attracting
and repelling sets for $g,g^{-1}$ then we shall always require additionally that
$$
 \mathcal{A}(g)\bigcap\big(\mathcal{R}(g)\bigcup\mathcal{A}(g^{-1})\big)=
 \mathcal{A}(g^{-1})\bigcap\big(\mathcal{R}(g^{-1})\bigcup\mathcal{A}(g)\big)=\emptyset.
$$
\end{remark}

\begin{definition} \label{def:dominated} 
If $g,h$ are very proximal elements with given associated repelling and attracting neighborhoods. We will say that {\it{$g$} is dominated by $h$} if $\mathcal{R}(g^{\eta}) \subset \mathcal{R}(h^{\eta}), \mathcal{A}(g^{\eta}) \subset \mathcal{A}(h^{\eta})$ for $\eta \in \pm 1$. 
\end{definition}
These notions depend not only on the elements themselves but on the specific choice of attracting and repelling neighborhoods, but when these neighborhoods are clear, we will suppress them from the notation.

Definition \ref{def:contraction} is stated in terms of the topological dynamics of the action of a single projective transformation $g \in \PGL_n(k)$ on the projective space $\PP(k^n)$.  A fundamental idea to the whole theory, which was fully developed in \cite{BG:Dense_Free}, \cite{BG:Topological_Tits} is that contraction can be alternatively expressed also in metric, or in algebraic terms. Metrically, in terms of the Lipschitz constant of $g$ on an open set away from the repelling hyperplane. Algebraically, in terms of the singular values of the corresponding matrix.  The equivalence between the different notions is quantitative.
\begin{lemma} \label{lem:equiv_contraction}
Let $F$ be a local field, $c_1,c_2$ the constants given in Lemma \ref{fix} and $g \in \GL_n(F)$. Denote by $[g]$ the image of $g$ in $\PGL_n(F)$ and by $g = kak'$ its Cartan decomposition, with $a = diag(a_1,a_2,\ldots, a_n), a_1 \ge a_2 \ge \ldots a_n > 0$. Set $H = \overline{\Span \{k'^{-1}(e_i)\}_{i=2}^{n} }, p = \overline{k e_1}$. Then there exists a constant $c > 0$ depending only on the field such that for any $0<\epsilon < \frac{1}{4}$ we have:
\begin{itemize}
\item If $\frac{a_1}{a_2} \ge \frac{1}{\epsilon^2}$ then $[g]$ is $\epsilon$-contracting with $(H, a)$ as a repelling hyperplane and an attracting point. Moreover $[g]$ is $\frac{\epsilon^2}{r^2}$-Lipschitz outside the $r-$neighborhood of $H$. 
\item Assume that the restriction of $[g]$ to some open set $O \subset \PP(F^n)$ is $\epsilon$-Lipschitz then $\frac{a_1}{a_2} \ge \frac{1}{2 \epsilon}$ and $[g]$ is $c \sqrt{\epsilon}$ contracting. 
\item If $[g]$ is $(r,\epsilon)$-contraction for $r>c_1\epsilon$, then it is $c\frac{\epsilon^2}{d^2}$-Lipschitz outside the $d$-neighborhood of the repelling hyperplane.
\end{itemize}
\end{lemma}

Using proximal elements, one constructs free groups with the following variant of the classical
ping-pong lemma.
\begin{lemma}\label{lem:ping-pong}
Suppose that $\{ g_i\}_{i\in I} \subset \PGL_n(k)$ is a set of very proximal elements, each
associated with some given attracting and repelling sets for itself and for its inverse. Suppose
that for any $i\neq j,~i,j\in I$ the attracting set of $g_i$ (resp. of $g_i^{-1}$) is disjoint
from both the attracting and repelling sets of both $g_j$ and $g_j^{-1}$. Then the $g_i$'s form a
free set, i.e. they are free generators of a free group.
\end{lemma}
This lemma calls for several important definitions that will play a central role in this paper. Our general setting is somewhat more general than usual, because we work with countable linear groups that are not necessarily finitely generated. 

Let $(K,v)$ be a complete valued field and $\Gamma < \PGL_n(K)$ a countable group. For every $\Delta < \Gamma$ let us denote by $(k(\Delta), v(\Delta))$ the closed subfield that is generated by the matrix coefficients of elements in $\Delta$. Even though $K$ itself will typically not be a local field , we do assume that $(k(\Delta),v(\Delta))$ is a local field whenever $\Delta$ is finitely generated.
\begin{definition}
A finite list of elements $\{g_i \in \Gamma\}_{i \in I}$ will be called a {\it{ping-pong}} or a {\it{Schottky}} tuple, and the group that they generate $\Delta = \langle g_i \ | \ i \in I \rangle$ will be called a {\it{Schottky subgroup}} if they satisfy the conditions of Lemma \ref{lem:ping-pong} on $\PP(k^n)$ for every intermediate local field $k(\Delta) < k < K$. Such a tuple, as well as the group it generates, will be referred to as {\it{spacious Schottky}} if there exists an additional element $g \in \Gamma$ such that $\{g_i \ | \ i \in I \} \bigcup \{g\}$ is still a Schottky tuple.  Finally we will call $\Delta < \Gamma$ {\it{locally Schottky}} or {\it{locally spacious Schottky}} if every finitely generated subgroup of $\Delta$ is such. We denote by $X = X(\Gamma)$ the collection of locally spacious Schottky subgroups of $\Gamma$. When $\Gamma$ is an abstract group, with an action on a projective space given by a representation $\rho: \Gamma \rightarrow \PGL_n(k)$, we will denote by $X^{\rho}(\Gamma)$ the collection of subgroups whose image under $\rho$ is locally spacious Schottky. 
\end{definition}
One example of a locally spacious Schottky group is just an {\it{infinite Schottky group}}. By definition this is a subgroup with a given infinite generating set $\Delta = \langle g_i  \ | \ i \in I \rangle$ such that for every finite $J \subset I$ the tuple $\{g_i \ | \ i \in J\}$ is a ping-pong tuple, and hence $\Delta_J = \langle g_j \ | \ j \in J \rangle$ plays ping-pong on $\PP(k_J)$ with $k_J = k(\Delta_J)$. The local fields $\{k_J\}$, and corresponding projective spaces $\{\PP(k_J)\}$ form a direct system inside $K$ and $\PP(K)$ respectively. Though note that in our current setting there is no canonical ping-pong playground where the generators play all together. The following lemma is easy, we leave its proof to the readers. 
\begin{lemma} \label{lem:subsch}
Let $\Delta = \langle \eta_1, \eta_2, \ldots, \eta_m, \zeta \rangle \in X(\Gamma)$ be a spacious Schottky subgroup. Let $\{e \ne w_j \ | \ j \in \N\}$ be any ordered set of nontrivial elements, possibly containing repetitions,  in the group $\langle \eta_1, \eta_2, \ldots, \eta_m \rangle$. Then $\{\zeta^{-i} w_i \zeta^{-i} \ | \ i \in \N \}$ is an infinite Schottky tuple.
\end{lemma}

\begin{definition} \label{def:Mflag}
Let $K$ be any field. A pair $\overline{\alpha} = \left(\overline{H}^{+},\overline{v}^{-} \right) $ with $\overline{H}^{+} < \PP(K^n)$ a hyperplane and $\overline{v}^{-} \in \overline{H}^{+}$ a point will be called a {\it{minimal flag}} or an {\it{M-flag}} for short. We denote by $\MF(K^n)$ the (projective) variety of all such M-flags. We will say that two M-flags $\overline{\alpha}_1, \overline{\alpha}_2$ {\it{touch each other}} if either $\overline{v}^{-}_{1} \in \overline{H}^{+}_{2}$ or $\overline{v}^{-}_{2} \in \overline{H}^{+}_{1}$. A collection of M-flags $\left \{\alpha_i = \left( \overline{H}^{+}_i,\overline{v}^{-}_i \right) \ | \ i  \in I \right\}$ is said to be {\it{in general position}} if the following conditions are satisfied: 
\begin{itemize}
\item $\bigcap_{j \in J} \overline{H}^{+}_{j} = \emptyset,$ for every $J \subset I$ with $|J|=n$.  
\item $\Span \{\overline{v}^{-}_{j} \ | \ j \in J \} := \overline{\Span \left\{v^{-}_{j} \ | \ j \in J \right\}} = \PP(K^n),$ for every $J \subset I$ with $|J|=n$.
\end{itemize}
\end{definition}
Note that the $\pm$ indices in this definition do not really play any role, they are only there in anticipation of the following:
\begin{definition} \label{def:Mping}
If $k$ is a local field and $g \in \PGL_n(k)$ is a very proximal element then we can associate with it an M-flag $\alpha(g) = \left(\overline{H}_{g}, \overline{v}_{g^{-1}}\right)$. We will say that a ping pong tuple $(g_1,g_2, \ldots, g_m)$ is {\it{in general position}} if both $\{\overline{\alpha}(g_i) \ | \ 1 \le i \le m \}$ and $\{\overline{\alpha}(g_i^{-1}) \ | \ 1 \le i \le m \}$ are. 
\end{definition} 
\begin{lemma} \label{lem:NoTouch}
If $\{\alpha_i \in \MF(K^n) \ | \ 1 \le i \le m\}$ is in general position with $m \ge 2n-1$, then no $\alpha \in \MF(K^n)$ can touch simultaneously all of the M-flags $\{\alpha_{i}, \ 1 \le i \le m\}$. 
\end{lemma}
\begin{proof}
Let $I = \{1,2,\ldots, 2n-1\}$ and assume that $\alpha_i = \left(\overline{H}_{i}^{+}, \overline{v}_{i}^{-} \right)$ and $\alpha = \left(\overline{H}^{+}, \overline{v}^{-} \right)$. Set $J = \{i \in I \ | \overline{v}^{-} \not \in \overline{H}^{+}_i \}$. Since $\overline{v}^{-} \in \bigcap_{i \in I \setminus J} \overline{H}^{+}_{i} \ne \emptyset$, the first condition in the general position assumption implies that $|J| \ge n$. Now the second condition implies that $\Span \left\{ \overline{v}_{j}^{-} \ | \ j \in J \right \} = \PP(K^n)$ so that $\overline{v}_{i_0}^{-} \not \in \overline{H}^{+}$ for some $i_0 \in J$. Thus the two M-flags $\alpha, \alpha_{i_0}$ fail to touch and the lemma is proved.
\end{proof}

\begin{definition}
A linear representation $\rho: \Gamma \rightarrow \GL_n(k)$ is called {\it{strongly irreducible}} if one of the following equivalent conditions hold:
\begin{itemize}
\item $\rho(\Gamma)$ does not preserve a finite collection of nontrivial proper subspaces. 
\item $\rho(\Delta)$ is irreducible for every finite index subgroup $\Delta < \Gamma$. 
\item $\left( \overline{\rho(\Gamma)}^{Z} \right)^{(0)}$ is irreducible. 
 \end{itemize}
 \end{definition}
We leave the verification of the equivalence to the reader. Strongly irreducible representations of groups of almost simple type offer a lot of flexibility, a fact which we attempt to capture in the next two lemmas. 
\begin{lemma} \label{lem:escape}
Let $\Gamma < \GL_n(K)$ be a group of almost simple type, $G = \overline{\Gamma}^Z$ its Zariski closure with $G^{(0)} = H^m$ as in Lemma \ref{lem:linear_ast}. Let $\rho: G \rightarrow \GL_m(K)$ be a strongly irreducible representation defined over $K$. Then $\rho(N)$ is irreducible for every $\trivgp \ne N \lhd \Gamma$. In particular given a vector $0 \ne v \in K^m$ and a hyperplane $V < K^m$ there is some $n \in N$ such that $\rho(n)(v) \not \in V$.
\end{lemma}
\begin{proof}
By the equivalence of the conditions in the previous definition $\rho(G^{(0)})$ is irreducible. By Lemma \ref{lem:linear_ast} $\overline{N}^Z > G^{(0)}$ so that $\rho(N)$ is irreducible as well. Since $\rho(N)$ is irreducible we have $$\Span\{\rho(n) v \ | \ n \in N \} = K^m$$ so this set cannot be contained in the proper subspace $V$.
\end{proof}
\begin{lemma} \label{lem:space}
Let $k$ be a local field, $m \in \N$ and $\Gamma < \GL_n(k)$ a strongly irreducible group of almost simple type. Assume that $\Gamma$ contains a very proximal element $g$. Then $\Gamma$ contains a Schottky tuple $(\eta_1, \eta_2, \ldots, \eta_m)$ in general position. More generally if $(\gamma_1,\gamma_2, \ldots, \gamma_l)$ is a spacious Schottky tuple, then one can find a spacious Schottky tuple $(\eta_1, \eta_2, \ldots, \eta_m)$ in general position, such that $(\gamma_1,\gamma_2, \ldots, \gamma_l, \eta_1, \eta_2, \ldots, \eta_m)$  is still spacious Schottky. 
\end{lemma}
\begin{proof}
We start with the first statement, namely with the special case $l=0$. Assume by induction that we already have a spacious Schottky tuple $\{\beta_i \ | \ i \in I \}$ whose corresponding attracting points and repelling hyperplanes satisfy all of the conditions implied so far. Namely:
\begin{itemize}
\item $\dim \left(\bigcap_{j \in J} \overline{H}^{+}_{j} \right) = n-|J|-1, \ \forall J \subset I$ 
\item $\dim(\Span\{\overline{v}^{-}_{i} \ | \ j \in J \}) = |J|-1, \ \forall J \subset I, |J| \le n$
\item $\overline{v}^{\eta}_{i} \not \in \overline{H}^{\epsilon}_{j}$ whenever either $i \ne j$ or $\epsilon = \eta$
\end{itemize}
Here $\dim$ denotes projective dimension, and we have adopted the convention that a negative dimension corresponds to an empty set. As the basis of the induction we can take $\beta_1$ to be any conjugate of $g$.

Now let us fix a finite set of points $A \subset \PP(k^n)$ such that $$\{\overline{v}_{i}^{\eta} \in A, \  | \ i \in I, \eta \in \pm 1\} \subset A$$ and also $A \bigcap \left(\bigcap_{j \in J} \overline{H}_{j}^{\eta} \right) \ne \emptyset,$ for every subset $J \subset I$ with $|J| < n,$ and every $\eta \in \{\pm 1\}$. Since $\Gamma$ is strongly irreducible we can find an element $\gamma \in \Gamma$ subject to the following two conditions:
\begin{itemize}
\item $A \bigcap \gamma \overline{H}_g^{\epsilon} = \emptyset, \ \forall \epsilon \in \pm 1$. 
\item $\gamma \overline{v}_{g}^{-} \not \in \bigcup_{\substack{J \subset I \\ |J| < n}}\Span \{\overline{v}_{j}^{-} \ | \ j \in J\}.$
\end{itemize}

Setting $\beta_{p+1} = \gamma g \gamma^{-1}$ we have $\overline{H}_{\beta_{p+1}}^{\epsilon} = \gamma \overline{H}_{g}^{\epsilon}$ and  $\overline{v}_{\beta_{p+1}}^{\epsilon} = \gamma \overline{v}_{g}^{\epsilon}, \ \forall \epsilon \in \pm 1$. After $m$ such steps we obtain a collection of very proximal elements whose attracting points and repelling hyperplanes are subject to all of the desirable conditions mentioned above. We obtain the desired ping-pong tuple by setting $\eta_i = \beta_i^N$ for a high enough value of $N$.

Now assume $l \ne 0$ and let $g$ be the proximal element such that $(\gamma_1,\gamma_2, \ldots, \gamma_l,g)$ is Schottky. Let us construct $(\eta_1', \eta_2', \ldots, \eta_n')$ as a spacious Schottky tuple in general position, just as we did in the last paragraph. By taking care we may assume that these were constructed in such a way that $(g^n, \eta'_1, \eta'_2, \ldots, \eta'_n)$ is still Schottky. Now setting $\eta_i = g \eta_i' g^{-1}$ we obtain a Schottky tuple in general position $(\eta_1,\eta_2,\ldots, \eta_n)$ subject to the additional condition that $\Ac(\eta_i^{\pm}) \subset \Ac(g), \Rc(\eta_i^{\pm}) \subset \Rc(g)$. This  implies that $(\gamma_1,\gamma_2, \ldots, \gamma_l, \eta_1, \eta_2, \ldots, \eta_l)$ is as required. 
\end{proof}

The main ingredient in the method we use for generating free subgroups is a projective representation whose image contains contracting elements and acts strongly irreducibly. The following theorem is a particular case of Theorem 4.3 from \cite{BG:Topological_Tits}. Note that a similar statement appeared also earlier in \cite{MS:Maximal}.

\begin{theorem}\label{thm:good-representation}
Let $F$ be a field and $\mathbb{H}$ an algebraic $F$-group for which the connected component $\mathbb{H}^{\circ}$ is not solvable, and let $\Psi<\Gamma <\mathbb{H}$ be Zariski dense subgroups with $\Psi$ finitely generated and $\Gamma$ countable. Assume that $\Psi$ contains at least one element of infinite order. Then we can find a valued field $(K,v)$, an embedding $F \hookrightarrow K$, an integer $n$ and a strongly irreducible projective representation $\rho :\mathbb{H}(K)\rightarrow \PGL_{n}(K)$ defined over $K$ with the following properties. 
\begin{enumerate}
\item $(k(\Delta),v(\Delta))$ is a local field for every finitely generated subgroup $\Delta < \Gamma$.  
\item There exits an element $\psi \in \Psi$, such that $\rho(\psi)$ is a very proximal element on $\PP(k^n)$ for some parameters $(\epsilon,r)$ satisfying the stronger condition appearing in Lemma \ref{fix} for every intermediate local field $k(\Psi) < k < K$. 
 \end{enumerate}
\end{theorem}
Where as above, $(k(\Delta),v(\Delta))$ denotes the closed subfield of $K$ generated by all matrix coefficients of $\rho(\Delta)$. 

\subsection{primitivity for linear groups of almost simple type}
\begin{theorem} \label{thm:ast1}
Let $\Gamma$ be a countable linear group of almost simple type, containing at least one element of infinite order. Then $\Gamma$ is primitive. 
\end{theorem} 
Let $\mathbb{H} = \overline{\Gamma}^{Z}$ be the  Zariski closure. By dimension considerations we can find a finitely generated subgroup $\Psi < \Gamma$ with the same Zariski closure. Obviously $\Psi$ contains an element of infinite order. We now apply Theorem \ref{thm:good-representation} and fix once and for all the data that this theorem yields. In fact we will just identify $\Gamma$ with its image under this representation $\rho(\Gamma)$ in order to avoid cumbersome notation. Once we have fixed this representation we will denote by $X(\Gamma) = X^{\rho}(\Gamma)$ the collection of all subgroups whose image under this fixed representation is locally spacious Schottky. With all this notation in place Theorem \ref{thm:ast1} will follow from the following, slightly more general statement. 

\begin{theorem} \label{thm:ast2}
Any finitely generated spacious Schottky subgroup $D < \Gamma$ is contained in a maximal core-free subgroup $D < M < \Gamma$. 
\end{theorem}
\begin{proof} 
By Proposition \ref{prop:cfgpd} It suffices to construct a subgroup $D < \Delta < \Gamma$ which is prodense and cofinitely generated. As a feature of the proof the group $\Delta$ we construct will be Schottky, and in particular $\Delta \in X(\Gamma)$. 

Let $(\delta_1, \delta_2, \ldots, \delta_m)$ be a Schottky generating set for the subgroup $D$. Let $D < D' < \Gamma$ be a larger (not necessarily Schottky) finitely generated subgroup that has the same Zariski closure as $\Gamma$. Letting $k = k(D')$ be the local field generated by the matrix coefficients of $D'$, this implies that $\rho(\Gamma) \bigcap \GL_n(k)$ is strongly irreducible. Appealing to Lemma \ref{lem:space} we can construct a spacious Schottky tuple $(\delta_1, \delta_2, \ldots, \delta_m, \sigma_1,\ldots, \sigma_{2n-1}, \zeta)$ with $\Sigma = \langle \sigma_1, \ldots, \sigma_2,\ldots, \sigma_{2n-1}, \zeta \rangle$, a Schottky subgroup whose generators are in general position, as in Definition \ref{def:Mflag}. The last element is denoted differently just because it will play a separate role in the proof. We set $\Delta_0 :=D, k_0 = k(\langle \Delta_0, \Sigma \rangle)$ and in $\PP(k_0^n)$ we denote by $\overline{H}_i^{\pm} = \overline{H}_{\sigma_i}^{\pm}, \overline{v}_i^{\pm} = \overline{v}_{\sigma_i}^{\pm}, \overline{H}^{\pm} = \overline{H}_{\zeta}^{\pm}, \overline{v}^{\pm} = \overline{v}_{\zeta}^{\pm}$ the attracting points and repelling hyperplanes of the corresponding ping pong game, all defined over $k_0$. We did not name the attracting points and repelling neighborhoods of the $\delta_i$'s, as we will not refer to them explicitly.

Let us use odd indices $\{ \gamma_1 \langle \langle n_1 \rangle \rangle,  \gamma_3 \langle \langle n_3 \rangle \rangle,  \gamma_5 \langle \langle n_5 \rangle \rangle, \ldots \}$ to list all cosets of all these normal subgroups that are generated by one, nontrivial, conjugacy class. Since any nontrivial normal subgroup contains one of these, the collection of these cosets forms a basis for the normal topology on $\Gamma$. So a subgroup intersecting all of them nontrivially will be prodense. Similarly we use even indices $(\Sigma \theta_2 \Sigma, \Sigma \theta_4 \Sigma, \ldots)$ to enumerate the nontrivial double-cosets of $\Sigma$ in $\Gamma$. A subgroup $\Delta < \Gamma$ that has a nontrivial intersection with all of these double cosets will be cofinitely generated by virtue of the fact that $\Gamma = \langle \Delta, \Sigma \rangle$. 

We will construct an infinitely generated Schottky group, 
$$\Delta = \langle \delta_1, \delta_2, \ldots, \delta_m, \eta_1, \eta_2, \ldots \rangle $$
such that $\eta_i \in \gamma_i \langle \langle n_i \rangle \rangle$ for every odd $i$ and $\eta_i \in \Sigma \theta_i \Sigma$ for every even $i$. This will be done by induction on $i$, with $\Delta_i := \langle D, \eta_1, \eta_2, \ldots, \eta_i \rangle$, the group constructed at the $i^{th}$ step and $k_i = k(\langle \Delta_i, \Sigma \rangle)$ the corresponding local subfield of $K$. We obtain a sequence of local fields $k_0 < k_1 < k_2 < \ldots < K$ with corresponding, direct sequence of projective spaces $\PP(k_0^n) \subset \PP(k_1^n) \subset \PP(k_2^n) \subset  \ldots \subset \PP(K^n)$. The generators of $\Delta_i$ will form a ping pong tuple in their action on $\PP(k_i^n)$. By Lemma \ref{lem:equiv_contraction}, the canonical attracting and repelling neighborhoods of the ping pong players $\eta_i$ will be of the form $\left(\overline{H}_{\eta_i}^{\pm} \right)_{\epsilon_i}, \left(\overline{v}_{\eta_i}^{\pm} \right)_{\epsilon_i}$. These are all defined over $k_i$, in the sense that $v_i \in k_i^n$, $H_i$ has a basis consisting of $(n-1)$ vectors in $k_i^n$ and $\epsilon$ is determined, by the singular values of $\eta_i$ which are inside $k_i$. Thus the same elements will form a ping-pong tuple also on $\PP(k_j)$ for any $j > i$. The attracting points and repelling hyperplanes in the extended vector space are obtained by an extension of scalars and $\epsilon_i$ remains the same. 

Assume we constructed $\Delta_{\ell} = \langle\delta_1, \ldots, \delta_m, \eta_1, \eta_2, \ldots, \eta_{\ell-1} \rangle < \PGL_n(k_{\ell-1})$ satisfying:
\begin{itemize}
\item $\eta_i \in \gamma_{i} \langle \langle n_i \rangle \rangle$ if $i$ is odd and $\eta_i \in \Sigma \theta_i \Sigma$ if $i$ is even. 
\item $\eta_{i} = \zeta^{i} q_{i} \zeta^{-i}$ with $q_{i}$ very proximal and dominated by some nontrivial element in $\langle \sigma_1, \sigma_2, \ldots, \sigma_{2n-1} \rangle < \Sigma$ for every $1 \le i < \ell$. 
\end{itemize}
The second condition guarantees that $\Delta_{\ell-1}$ is Schottky by Lemma \ref{lem:subsch}. 

When $\ell$ is odd, we are looking for an element $\eta_{\ell} \in \gamma_{\ell} \langle \langle n_{\ell} \rangle \rangle$. We extend scalars to $K$, but by abuse of notation we will identify $L_i^{\pm} < k_i^n$ with their scalar extensions $L_i^{\pm} \otimes_{k_{\ell-1}} K < K^n$ and similarly for $w_{i}^{\pm} \in k_{\ell-1}^n \subset K^n$. By Lemma \ref{lem:escape} we find $n_1,n_2,n_3 \in \langle \langle n_{\ell} \rangle \rangle$ such that $n_1^{\epsilon} \overline{v}_{1}^{+} \not \in \overline{H}_{1}^{-}, n_2 \gamma_{\ell} \overline{v}^{+} \not \in \overline{H}^{-}, n_2^{-1} \overline{v}^{+} \not \in \gamma \overline{H}^{-}, n_3^{\epsilon} \overline{v}_{1}^{-} \not \in \overline{H}_{1}^{+}, \ \forall \epsilon \in \{\pm 1\}$. Let $k_{\ell} = k(\langle \Delta_{\ell-1}, \Sigma, n_1,n_2,n_3 \rangle)$ and consider the element 
$$\eta_{\ell} = \zeta^{\ell} q_{\ell} \zeta^{-\ell} = \zeta^{\ell} \sigma_{1}^{N}n_3 \sigma_{1}^{-N} \zeta^{-\ell} n_2 \gamma_{\ell} \zeta^{\ell} \sigma_{1}^{-N}n_1 \sigma_{1}^{N} \zeta^{-\ell} \in \gamma_{\ell} \langle \langle n_{\ell} \rangle \rangle,$$  
where the auxiliary element $q_{\ell}$ is also defined by the above equation. Following the dynamics of the action of this element on $\PP(k_{\ell}^n)$ we see that, for $N$ large enough, $q_{\ell}$ is very proximal and dominated by $\sigma_1$. Thus by Lemma \ref{lem:subsch} the tuple $(\delta_1, \delta_2, \ldots, \delta_m, \eta_1, \eta_2, \ldots, \eta_{\ell}) = (\delta_1, \delta_2, \ldots, \delta_m, \zeta q_1 \zeta^{-1}, \ldots, \zeta^{\ell} q_{\ell} \zeta^{-\ell})$ is Schottky. Group theoretically it is easy to verify that $\eta_{\ell}$ thus defined belongs to the desired coset $\gamma_{\ell} \langle \langle n_{\ell} \rangle \rangle$.  

Next, assume that $\ell$ is even. Now our goal is to construct a ping-pong player of the form $\eta_{\ell} = \zeta^{\ell} q_{\ell} \zeta^{-\ell} \in \Sigma \theta_{\ell} \Sigma$. By Lemma \ref{lem:NoTouch} we can find some $i = i_{\ell}$ such that the M-flag $\alpha(\sigma_i)$ does not touch the M-flag $\theta_{\ell} \alpha(\sigma_1^{-1})$. Explicitly this means that $\overline{v}_{i}^{-} \not \in \theta_{\ell} \overline{H}_{1}^{-}$ and $\theta_{\ell} \overline{v}_{1}^{+} \not \in \overline{H}_{i}^{+}.$ These are exactly the conditions needed in order to ensure that, for a high enough value of $n = n_l$ the element $q_{\ell} := \sigma_{i}^{n} \theta_{l} \sigma_1^{n}$ is dominated by the element $\sigma_i^n \sigma_1^n \in \Sigma$. Now set $\eta_{\ell} : = \zeta^{l} \sigma_{i}^{n} \theta_{\ell} \sigma_1^{n} \zeta^{- \ell} \in \Sigma \theta_{\ell} \Sigma$. This concludes the even step of the induction and completes the proof.  
\end{proof}

\section{Counting maximal subgroups of $\SL_n(\Z)$}

When restricting the attention to $\SL_n(\Z),~n\ge 3$, one can make use of the arithmetic structure and the abundance of unipotent elements to produce $2^{\aleph_0}$ different maximal subgroups.
We follow the argument of \cite{GM:max}.
 
\subsection{Projective space}
Let $n \ge 3$. By $\BP(k^{n})$ we denote the $(n-1)$-dimensional real projective space endowed with the metric defined in Section \ref{sec:pt}. For every $0 \le k \le n-1$, the set $\mathbb{L}_k$ of $k$-dimensional subspaces of $\BP(k^{n})$ can be endowed with the metric defined by 
$$
 \dist_{\mathbb{L}_k}(L_1,L_2):=\max\{d(k^{n})(x,L_i) \mid x \in L_{3-i} \forall \ 1 \le i \le 2\}
 $$ 
 for every $L_1,L_2 \in \mathbb{L}_k$. Note that 
$\mathbb{L}_k$ is naturally homeomorphic to the Grassmannian $\mathrm{Gr}(k+1,\R^n)$. 


\subsection{Unipotent elements}
\begin{definition}[Rank 1 unipotent elements]\label{def - rank 1 unipotent} We say that a unipotent element $u$ has rank 1 if $\text{rank}(u-\textrm{I}_n)=1$. The point $p_u\in \BP(k^{n})$ which is induced by the euclidean line ${\{ux-x\mid x \in \R^n\}} $ is called the point of attraction of $u$.  The $(n-2)$-dimensional subspace  $L_u\subseteq \BP(k^{n})$ which is induced by the euclidean $(n-1)$-dimensional space ${\{x \in \R^n \mid ux=x\}}$ is called the fixed hyperplane of $u$. The set of rank-1 unipotent elements in $SL_n(\Z)$ is denoted by $\mathcal{U}$. 
\end{definition}

The following two lemmas follow directly from the definition of $\mathcal{U}$ and are stated for future reference. 

\begin{lemma}[Structure of unipotent elements]\label{lemma - structure of unipotent} 
The set $\mathcal{U}$ can be divided into equivalence classes in the following way: $u,v \in \mathcal{U}$ are equivalent if 
there exist non-zero integers $r$ and $s$ such that $u^s=v^r$.
The map $u \mapsto (p_u,L_u)$ is a bijection between equivalence classes in $\mathcal{U}$ and the set of pairs $(p,L)$ where $p \in \BP(k^{n})$ is a rational point and $ L \subseteq \mathbb{L}_{n-2}$ 
is an $(n-2)$-dimensional rational subspace which contains $p$.
\end{lemma}

\begin{lemma}[Dynamics of unipotent elements]\label{lemma - dynamic of unipotent} Let $u \in \mathcal{U}$.  For every $\varepsilon>0$ and every $\delta>0$ there exists a constant $c$ such that if 
$m \ge c$
and $v=u^m $, then  $v^{k}(x) \in (p_u)_{\varepsilon}$ for every $x \in \BP(k^{n})\setminus (L_u)_{\delta}$ and every  $k \ne 0$. Note that the previous lemma implies that $p_u=p_v$ and $L_u=L_v$.
\end{lemma}

\subsection{Schottky systems}

\begin{definition} Assume that $\mathcal{S}$ is a non-empty subset of $\mathcal{U}$ and $\mathcal{A} \subseteq \mathcal{R}$ are closed subsets of $\BP(k^{n})$. We say that $\mathcal{S}$ is a Schottky set with respect to the attracting set  $\mathcal{A}$ and the repelling set $\mathcal{R}$ and call the triple 
$(\mathcal{S},\mathcal{A},\mathcal{R})$ a Schottky system if for every $u \in \mathcal{S}$ there exist two positive
numbers  $\delta_u \ge \epsilon_u$ such that the following properties hold:
\begin{enumerate}
\item $u^{k}(x) \in (p_u)_{\varepsilon_u}$ for every $x \in \BP(k^{n})\setminus (L_u)_{\delta_u}$ and every  $k \ne 0$;
\item If $u \ne v \in \mathcal{S}$ then $(p_u)_{\varepsilon_u} \bigcap (L_v)_{\delta_v}=\emptyset$;
\item $\bigcup_{u \in \mathcal{S}}(p_u)_{\varepsilon_u} \subseteq \mathcal{A}$;
\item $\bigcup_{u \in \mathcal{S}}(L_u)_{\delta_u} \subseteq \mathcal{R}$.
\end{enumerate}
\end{definition} 
\begin{definition} 
The Schottky system $(\mathcal{S},\mathcal{A},\mathcal{R})$ is said to be profinitely-dense if $\mathcal{S}$ generates a profinitely-dense subgroup of $\SL_n(\Z)$.
We say that the Schottky system $(\mathcal{S}_+,\mathcal{A}_+,\mathcal{R}_+)$ contains the Schottky system $(\mathcal{S},\mathcal{A},\mathcal{R})$ if 
$\mathcal{S}_+ \supseteq \mathcal{S}$, $\mathcal{A}_+ \supseteq \mathcal{A}$ and $\mathcal{R}_+ \supseteq \mathcal{R}$. 
\end{definition}

\begin{lemma}\label{adding} Let $(\mathcal{S},\mathcal{A},\mathcal{R})$ be a Schottky system. 
Assume that $[p]_\epsilon \bigcap \mathcal{A}=\emptyset$ and $[L]_\delta \bigcap \mathcal{R}=\emptyset$ where $\delta \ge\epsilon >0$ and 
$p$ is a rational point which is continued in a rational 
subspace $L \in \mathbb{L}_{n-2}$.  Denote $\mathcal{A}_+=\mathcal{A} \bigcup [p]_\epsilon$ and $\mathcal{R}_+=\mathcal{R} \bigcup [L]_\delta$. 
Then there exist $v \in \mathcal{U}$ with $p=p_v$, $L=L_v$ such that  $(\mathcal{S}_+,\mathcal{A}_+,\mathcal{R}_+)$ is  Schottky system which contains $(\mathcal{S},\mathcal{A},\mathcal{R})$ where $\mathcal{S}_+:=\mathcal{S}\bigcup\{v\}$.
\end{lemma}

\begin{proof} Lemma \ref{lemma - structure of unipotent} implies that there exists  $u \in \mathcal{U}$ such that $p_u=p$ and $L_u=L$.
Lemma \ref{lemma - dynamic of unipotent}  implies that there exists $m \ge 1$ such that $v:=u^m$ satisfies the required properties.  
\end{proof}

The following lemma is a version of the well known ping-pong lemma:

\begin{lemma}[Ping-pong]\label{lemma - ping pong} Let $(\mathcal{S},\mathcal{A},\mathcal{R})$ be a 
Schottky system. Then the natural homomorphism $*_{u \in \mathcal{S}}\langle u\rangle \rightarrow \langle \mathcal{S}\rangle $ is an isomorphism. 
\end{lemma}

An important ingredient for our methods is the following beautiful result:

\begin{theorem} \cite[Venkataramana]{Venkataramana:LC_completions} \label{Thm - Venkataramana} Let $\Gamma$ be a Zariski-dense subgroup of $\SL_n(\Z)$. Assume that
$u\in \mathcal{U} \bigcap \Gamma$, $v \in \Gamma$ is unipotent  and $\langle u,v\rangle \simeq \Z^2$. Then $\Gamma$ has finite index in $\SL_n(\Z)$. In particular, if  
$\Gamma$ is profinitely-dense then $\Gamma=\SL_n(\Z)$. 
\end{theorem}

Note that if $u,v \in \SL_n(\Z)\bigcap \mathcal{U}$ and $p_u=p_v$ then $(u-1)(v-1)=(v-1)(u-1)=0$ and in particular
$uv=vu$. 
Thus we get the following lemma:
\begin{lemma}\label{lemma - finding a good pair} Let $g \in \SL_n(\Z)$ and $u_1,u_2 \in \mathcal{U}$.
Assume that $p_{u_2}=gp_{u_1}$ and $L_{u_2} \ne gL_{u_1}$. Then $\langle u_1,g^{-1}u_2g \rangle \simeq \Z^2$.
\end{lemma}

\begin{lemma}\label{throwing} Assume that $g$ is an element of $\SL_n(\Z)$, $(\mathcal{S},\mathcal{A},\mathcal{R})$ is a profinitely-dense Schottky system, $\delta \ge \epsilon>0$, 
$p_1$ and $p_2$ are rational  points and $L_1$ and $L_2$ are rational  $(n-2)$-dimensional subspaces such that the following conditions hold:
\begin{enumerate}
\item $([p_1]_\epsilon  \bigcup [p_2]_\epsilon)\bigcap \mathcal{R}=\emptyset$ and $([L_1]_\delta \bigcup [L_2]_\delta)\bigcap \mathcal{A}=\emptyset$;
\item $[p_1]_\epsilon \bigcap [L_2]_\delta=\emptyset$ and $[p_2]_\epsilon \bigcap [L_1]_\delta=\emptyset$;
\item $p_1=gp_2$ and $L_1 \ne gL_2$.
\end{enumerate}
Denote $\mathcal{A}_+=\mathcal{A} \bigcup [p_1]_\epsilon \bigcup [p_2]_\epsilon$ and $\mathcal{R}_+=\mathcal{R} \bigcup [L_1]_\delta \bigcup [L_2]_\delta$.
Then there exists a set $\mathcal{S}_+ \supseteq \mathcal{S}$ such that $(\mathcal{S}_+,\mathcal{A}_+,\mathcal{R}_+)$ is  a Schottky system  which contains $(\mathcal{S},\mathcal{A},\mathcal{R})$
and $\langle \mathcal{S}_+,g \rangle=\SL_n(\Z)$. 
\end{lemma}

\begin{proof} For every $1 \le i \le 2$ choose $u_i\in \mathcal{U}$ such that
$p_{u_i}=p_i$ and $L_{u_i}=L_i$.
  Lemma \ref{lemma - finding a good pair} implies that $\langle u_1,g^{-1}u_2g \rangle \simeq \Z^2$.
Lemma \ref{lemma - dynamic of unipotent}  implies that there exists $m \ge 1$ such that $(\mathcal{S}_+,\mathcal{A}_+,\mathcal{R}_+)$ is  Schottky system
where $v_1:=u_1^m$, $v_2:=u_2^m$ and $\mathcal{S}_+:=\mathcal{S} \bigcup \{v_1,v_2\}$. Theorem \ref{Thm - Venkataramana} implies  
that $\langle \mathcal{S}_+,g \rangle=\SL_n(\Z)$. 
\end{proof}

\begin{definition} Let $1 \le k \le n$. A $k$-tuple $(p_1,\ldots,p_k) $ of projective points is called generic if  $p_1,\ldots,p_k$ span a $(k-1)$-dimensional
subspace of $\BP(k^{n})$. Note that the set of generic $k$-tuples of $\BP(k^{n})$ is an open subset of the product of $k$ copies of the projective space, indeed it is even Zariski open.
\end{definition}

\begin{theorem}[Conze-Guivarc'h, \cite{CG00}]\label{them - conze} Assume that $n \ge 3$ and that $\Gamma \le \SL_n(\R)$ is a lattice. Then $\Gamma$ acts minimally
of the set of generic $(n-1)$-tuples.  
\end{theorem}

\begin{corollary}\label{cor - conze} Assume that  $n \ge 3$ and $\Gamma \le \SL_n(\R)$  is a lattice. For  every $1 \le i \le 2$ let $p_i \in L_i \in \mathbb{L}_{n-2}$. Then for every positive numbers $\varepsilon$
and $\delta$ there exists $g \in \Gamma$ such that $gp_1 \in (p_2)_{\varepsilon}$ and
$gL_1 \in (L_2)_{\delta}$.
\end{corollary}

The proof of the following Proposition is based on the proof of the main result of \cite{AGS:gen}. 

\begin{proposition}\label{prop -pro-dense} Assume that $n \ge 3$ and $p \in L \in \mathbb{L}_{n-2}$.
Then for every $\delta \ge \epsilon >0$ there exists a finite subset $\mathcal{S}\subseteq \mathcal{U}$ such that  $(\mathcal{S},\mathcal{A},\mathcal{R})$ is a
profinitely-dense Schottky system  where $\mathcal{A}:=[p]_\epsilon$ and  $\mathcal{R}:=[L]_\delta$.
\end{proposition}

We will make use of the following well known:

\begin{lemma}\label{lem:profinitely-dense}
	Let $n\ge 3$ and $\Gamma \le \SL_n(\Z)$ be a subgroup such that $\Gamma$ projects on $\SL_n(\Z/4\Z)$ and $\SL_n(\Z/p\Z)$ for every odd prime $p$. Then $\Gamma $ is profinitely dense in  $\SL_n(\Z)$. 
\end{lemma}

We sketch a proof that was shown to us by Chen Meiri:

\begin{proof}
For $1 \le i \ne j \le n$ and $a \in \Z$ let $E_{i,j}(a)$ be the matrix with 1 on the diagonal, $a$ on the $(i,j)$-entry and zero elsewhere. For $a,b,m \in \Z$ we write $a\equiv_m b$ to indicate that $a$ is equal to $b$ modulo $m$. For every every prime $p$ and every $m \ge 1$, let $\pi_{p^m}:\SL_n(\Z)\rightarrow \SL_n(\Z/p^m\Z)$ be the reduction map. The following two facts are straightforward:
		
\begin{fact}\label{claim1}
	Let $p$ be an odd prime. If $A \equiv_p E_{i,j}(1)$ then for every $m \ge 1$, $A^{p^m} \equiv_{p^{m+1}} E_{i,j}(p^m)$.  
\end{fact}	

\begin{fact}\label{claim2}
	If $A \equiv_4 E_{i,j}(2)$ then for every $m \ge 1$, $A^{2^m} \equiv_{2^{m+2}} E_{i,j}(2^{m+1})$.  
\end{fact}	

\begin{clm}\label{claim3} Let $p$ be an odd prime and  $S \subseteq \SL_n(\Z)$. If $\pi_p(S) \supseteq \pi_{p}(\{E_{i,j}(1)\mid 1 \le i \ne j \le n\})$ then for every $m \ge 1$,   $\pi_{p^m}(\langle S\rangle)=\SL_n(\Z/p^m\Z)$.
\end{clm}
\begin{proof}
	The proof is by induction on $m$. The case $m=1$ is clear.  Assume that the claim holds for some $m \ge 1$. By \ref{claim1}, 
	$$
	\pi_{p^{m+1}}(\langle S\rangle) \supseteq \{\pi_{p^{m+1}}(E_{i,j}(p^m)) \mid 1\le i \le j \le n\}.
	$$ 
	The result follows since 
	$
	\SL_n(\Z/p^m\Z) / \SL_n(\Z/p^{m+1}\Z) 
	$
	 is isomorphic to $ \mathfrak{sl}_n(\Z/p\Z)$  and $\mathfrak{sl}_n(\Z/p\Z)$ is spanned by  $\{ge_{i,j}(1)g^{-1} \mid g \in \SL_n(\Z/p\Z),\ 1 \le i \ne j \le n\}$.
\end{proof}

Arguing similarly, one obtains:

\begin{clm}\label{claim4} Let $S \subseteq \SL_n(\Z)$. If $\pi_4(S) \supseteq \pi_{4}(\{E_{i,j}(1)\mid 1 \le i \ne j \le n\})$ then for every $m \ge 2$,  $\pi_{2^m}(\langle S\rangle)=\SL_n(\Z/2^m\Z)$.
\end{clm}

\begin{clm}\label{claim5} For every distinct primes $q,p_1,\ldots,p_k$ with $q$ odd and every $m \ge 1$, the group $\Gamma$ contains an element $A$ such that $\pi_q(A)=\pi_q(E_{i,j}(1))$ and  $\pi_{p_r^{m}}(A)=\pi_{p_r^{m}}(I_n)$, for every $1 \le r \le k$. 
\end{clm}
\begin{proof} For every prime $p$, $\PSL_n(\Z/p\Z)$ is simple. The Jordan-Holder theorem implies that $\Gamma$ projects onto $\SL_n(\Z/q\Z) \times \prod_{1 \le r \le k}\PSL_n(\Z/p_r\Z)$. Choose $1 \le l \le n$ distinct from $i$ and $j$. Then $\Gamma$ contains an element $B$ such that  $\pi_q(B)=\pi_q(E_{l,j}(1))$ and  $\pi_{p_r}(B)=\pm \pi_{p_r}(I_n)$, for every $1 \le r \le k$. Choose $C \in \Gamma$ such that $\pi_q(C)=\pi_q(E_{i,l}(1))$. Denote $D:=[C,B]$. Then,  $\pi_q(D)=\pi_q(E_{i,j}(1))$ and  $\pi_{p_r}(D)=\pi_{p_r}(I_n)$, for every $1 \le r \le k$. Choose $t$ such that $t(p_1\cdots p_k)^{m-1}\equiv_q 1$. Then $A=D^t$ is the required element. 
\end{proof}

In order to complete the proof of Lemma \ref{lem:profinitely-dense},
note that in view of the congruence subgroup property it is enough to prove that for every distinct primes $p_1,\ldots,p_k$ and every $m \ge 2$, $\Gamma$ projects onto $\prod_{1 \le r \le k}\SL_n(\Z/p_r^m)$.  Claim \ref{claim4} implies that for every $m \ge 1$, $\Gamma$ projects onto $\SL_n(\Z/2^m\Z)$. The result follows from Claims \ref{claim3} and \ref{claim5}.  
\end{proof}


\begin{proof}[Proof of Proposition \ref{prop -pro-dense}] We recall some facts about Zariski-dense and profinitely-dense subgroups.  For a positive integer $d \ge 2$ let 
$\pi_d:\SL_n(\Z)\rightarrow \SL_n(\Z/d\Z)$ be the modulo-d homomorphism and denote 
$K_d:=\ker \pi_d$. 

\begin{itemize}
\item[(a)] If $H \le \SL_n(\Z)$ and $\pi_p(H)=\SL_n(\Z/p\Z)$ for some odd prime $p$ then 
$H$ is Zariski-dense, \cite{We96} and \cite{lub:14aa}.  
\item[(b)] The strong approximation theorem  of Weisfeiler \cite{Weisfeiler:SAT} and Nori \cite{Nori:SAT} implies that if a subgroup $H$ of  $\SL_n(\Z)$ is Zariski-dense then there exists some positive integer $q$ such that $\pi_d(H)=\SL_n(\Z/d\Z)$ whenever $\gcd(q,d)=1$. 
 \end{itemize}

Fix $\delta \ge \epsilon>0$ and set $\mathcal{A}:=[p]_\epsilon$ and $\mathcal{R}:=[L]_\delta$. For every $1 \le i \le 2n^2-n$, fix a point $p_i$ belonging to an $(n-2)$-dimensional
subspace $L_i$ and positive numbers  $\delta_i \ge \epsilon_i>0$ such that the following two conditions hold:
\begin{itemize}
\item[(1)] $\bigcup_{1 \le i \le 2n^2-n}(p_i)_{\varepsilon_i}\subseteq \mathcal{A}$
 and  $\bigcup_{1 \le i \le 2n^2-n}(L_i)_{\delta_i}\subseteq \mathcal{R}$;
 \item[(2)] For every $1 \le i \ne j \le 2n^2-n$, $(p_i)_{\varepsilon_i} \bigcap (L_j)_{\delta_j}=\emptyset$.
\end{itemize}  
For every $1 \le i \ne j \le n$, let $e_{i,j}\in \SL_n(\Z)$ be the matrix with 1 on the diagonal and on the $(i,j)$-entry and zero elsewhere and
let $e_1,\ldots,e_{n^2-n}$ be an enumeration of the $e_{i,j}$'s. Denote the exponent of $\SL_n(\Z/3\Z)$ by $t$.
If $g_1,\ldots,g_{n^2-n}\in K_3$ and  $k_1,\ldots,k_{n^2-n}$ are positive integers then
$\pi_3(H_1)=\SL_n(\Z/3\Z)$ where $u_i:=g_ie_i^{tk_i+1}g_i^{-1}$ and $H_1:=\langle u_i \mid 1 \le i \le n^2-n \rangle$.  Note that for 
every $u \in \mathcal{U}$ and $g \in \SL_n(\Z)$, $p_{gug^{-1}}=gp_u$ and $L_{gug^{-1}}=gL_u$. Thus, Lemma \ref{lemma - dynamic of unipotent} and Corollary \ref{cor - conze} imply that it is possible to choose $g_i$'s and $k_i$'s
such that:
\begin{itemize}
\item[(3)]  $u_i^k(x) \in (p_i)_{\varepsilon_i}$ for every $1 \le i \le n^2-n$, every $x \not \in (L_i)_{\delta_i}$ and every $k \ne 0$.
\end{itemize}
In particular, $\{ u_1,\ldots, u_{n^2-n}\}$ is a Schottky set with respect to $\mathcal{A}$ and $\mathcal{R}$ which generates a Zariski-dense subgroup $H_1$. 

The strong approximation theorem implies that  there exists some positive integer $q$ such that $\pi_d(H_1)=\SL_n(\Z/d\Z)$ whenever $\gcd(q,d)=1$.
Denote the exponent of $\SL_n(\Z/q^2\Z)$ by $r$.
 As before, there exist $g_{n^2-n+1},\ldots,g_{2n^2-2n}\in K_{q^2}$ and positive integers $k_{n^2-n+1},\ldots,k_{2n^2-2n}$ such that the elements of the form $u_i:=g_ie_i^{rk_i+1}g_i^{-1}$ satisfy:
\begin{itemize}
\item[(4)] $\pi_{q^2}(H_2)=\SL_n(\Z/q^2\Z)$ where   $H_2:=\langle u_i \mid n^2-n+1 \le i \le 2n^2-2n \rangle$;
\item[(5)]   $u_i^k(x) \in (p_i)_{\varepsilon_i}$ for every $n^2-n+1 \le i \le 2n^2-2n$, every $x \not \in (L_i)_{\delta_i}$ and every $k \ne 0$. 
\end{itemize}

Denote $\mathcal{S}:=\{ u_1,\ldots, u_{2n^2-2n}\}$.  Lemma \ref{lem:profinitely-dense} implies that $\pi_d(\langle \mathcal{S} \rangle)=\SL_n(\Z/d\Z)$ for every $d \ge 1$. 
Thus, $(\mathcal{S},\mathcal{A},\mathcal{R})$ is the required profinitely-dense Schottky system.
\end{proof}


Zorn's lemma implies that every proper subgroup $H$ of $\SL_n(\Z)$ is contained in a maximal subgroup $M$
(Since $\SL_n(\Z)$ is finitely generated an increasing union of proper subgroups is a proper subgroup). 
If $H$ is profinitely-dense then so is $M$; hence $M$ should
have infinite index. Thus, Theorem \ref{thm counting} follows from the following proposition:

\begin{theorem} 
Let $n \ge 3$. There exist $2^{\aleph_0}$ infinite-index profinitely-dense subgroups of $\SL_n(\Z)$
such that the union of any two of them generates $\SL_n(\Z)$.
\end{theorem}
    
\begin{proof} For every non-negative integer $i$ fix a rational point $p_i$ belonging to a rational $(n-2)$-dimensional 
subspace $L_i$ and two numbers $\delta_i \ge \epsilon_i >0$ such that $[p_i]_{\varepsilon_i} \bigcap [L_j]_{\delta_j}=\emptyset$  for every $i \ne j$.
Let $\mathcal{A}$ and $\mathcal{R}$ be the closures of $\bigcup_{i \ge 0}(p_i)_{\epsilon_i}$ and $\bigcup_{i \ge 0}(L_i)_{\delta_i}$ respectively. Proposition \ref{prop -pro-dense} implies that there exists a finite subset $\mathcal{S}_0\subseteq \mathcal{U}$ such that $(\mathcal{S}_0,\mathcal{A}_0,\mathcal{R}_0)$ is a profinitely-dense
Schottky system where  $\mathcal{A}_0 = [p_0]_{\varepsilon_{0}}$ and $\mathcal{R}_0 = [L_0]_{\delta_0}$.  Lemmas \ref{lemma - structure of unipotent} and \ref{lemma - dynamic of unipotent} imply that for every $i \ge 1$ there are $u_{i,1},u_{i,2} \in \mathcal{U}$ such that:
\begin{enumerate}
\item $p_i=p_{u_{i,1}}=p_{u_{i,2}}$ and $L_{u_{i,1}} \ne L_{u_{i,2}} \subseteq (L_i)_{\delta_i}$ (hence, $\langle u_{i,1},u_{i,2}\rangle\cong \Z^2$);
\item $u_{i,j}^k(x) \in (p_i)_\varepsilon$ for every $1 \le j \le 2$, every $x \not \in (L_i)_{\delta_i}$ and every $k \ne 0$. 
\end{enumerate}
For every function $f$ from the positive integers to $\{0,1\}$ the set $\mathcal{S}_f:=\mathcal{S}_0\bigcup \{u_{i,f(i)\mid i\ge 1}\}$ is a Schottky set  with respect to the attracting set  $\mathcal{A}$ and the repelling set $\mathcal{R}$. If $f$ and $g$ are distinct functions then $\mathcal{S}_f \bigcup \mathcal{S}_g $ contains $\{u_{i,1},u_{i,2}\}$ for some $i \ge 1$ so Theorem \ref{Thm - Venkataramana} implies that 
$\langle \mathcal{S}_f \bigcup \mathcal{S}_g \rangle =\SL_n(\Z)$. 
\end{proof}

\section{Higher transitivity in negative curvature settings} \label{sec:ht}

\noindent
Our goal in this section is to prove theorem \ref{thm:rank_one_ht}

\subsection{Precise ping-pong dynamics}  \label{sec:ppd} 
The proof will proceed via the topological dynamics of the action of $\Gamma$ on $\PP := \PP(k^2)$ and on the limit set $L = L(\Gamma) \subset \PP$ (see Lemma \ref{lem:limit}). By a neighborhood of a point $p$ we mean any set containing $p$ in its interior. By a {\it{fundamental domain}} for the action of $\Gamma$ on an open invariant subset $Y \subset \PP$, we will always refer to an open subset $O \subset Y$ satisfying (i) $\gamma O \bigcap O=\emptyset, \ \forall \gamma\in\Gamma\setminus\{1\}$, (ii) $Y \subset \bigcup_{\gamma \in \Gamma} \gamma \overline{O}$ and (iii) $\partial O$ has an empty interior. 

Let $\gamma \in \SL_2(k)$ be a very proximal element. In our current rank one setting there is a new symmetry between the attracting and repelling neighbourhoods: the repelling hyperplane $\overline{H}^{+}_{\gamma}$ reduces to a single point and coincides with the attracting point of the inverse $\overline{v}^{-}_{\gamma}$. To emphasize this we slightly change the notation. We say that $\Omega_{\gamma}^{\pm} \subset \PP$ are attracting and repelling neighborhoods for a very proximal element $\gamma$ on $\PP$ if they are closed, disjoint neighborhoods of the attracting and repelling points $\overline{v}_{\gamma}^{\pm} \in \Omega_{\gamma}^{\pm}$ satisfying $\gamma \left(L \setminus \Omega^{-} \right) \subset \Omega^{+}$, (or equivalently $\gamma^{-1} \left( L \setminus \Omega^{+} \right) \subset \Omega^{-}$). We will further call such attracting and repelling neighborhoods {\it{precise}} if $O = L \setminus (\Omega^{+} \bigcup \Omega^{-})$ is a fundamental domain for the action of the cyclic group $\langle \gamma \rangle$ on $L \setminus \{\overline{v}^{+},\overline{v}^{-}\}$. 
\begin{lemma} \label{lem:limit} 
Let $k$ be a local field, and $\Gamma < \SL_2(k)$ a center free unbounded countable group. Assume $\Gamma$ neither fixes a point, nor a pair of points in $\PP(k^2)$. Then
\begin{enumerate}
\item \label{itm:exist_prox} $\Gamma$ contains a very proximal element. 
\item \label{itm:limit} There is unique minimal closed, $\Gamma$-invariant subset $L=L(\Gamma) \subset \PP(k^2)$. Moreover $L(\Gamma)$ is perfect, as a topological space. 
\item \label{itm:dense_pairs} The collection $\{(\overline{v}_{\gamma}^{+}, \overline{v}_{\gamma}^{-}) \ | \ \gamma \in \Gamma {\text{ very proximal}} \}$ is dense in $L(\Gamma)^2$. 
\item \label{itm:ess_free} If $e \ne \gamma \in \Gamma$ then $\Supp(\gamma) = \{x \in L(\Gamma) \ | \ \gamma x \ne x\}$ is an open dense subset of $L(\Gamma)$. 
\end{enumerate}
\end{lemma}
\begin{proof}
Since $\Gamma$ is unbounded it contains elements with an unbounded ratio between their singular values. Thus by Lemma \ref{lem:equiv_contraction}, for every $\epsilon > 0$, we can arrange for $g \in \Gamma$ such that both $g,g^{-1}$ are $\epsilon$-contractions. Let $h \in \Gamma$ be such that $h \overline{v}_{g}^{+} \ne v_{g}^{-}$ then $h g^n$ will be very proximal for a high enough value of $n$. 

Let $g$ be a very proximal element with attracting and repelling points $\overline{v}^{\pm}$. We know that $\lim_{n \arrow \infty} g^n(x) = \overline{v}^{+}, \forall x \ne \overline{v}^{-}$ and if $h \overline{v}^{-} \ne \overline{v}^{-}$ then $\lim_{n \arrow \infty} g^n h (\overline{v}^{-}) = \overline{v}^{+}$. So $L:=\overline{\Gamma (\overline{v}_{1}^{+})}$ is contained in every closed $\Gamma$-invariant set proving the first sentence of (\ref{itm:limit}). 

Now let $U^{\pm} \subset L(\Gamma)$ be two (relatively) open subsets. As the action of $\Gamma$ on $L(\Gamma)$ is clearly minimal we have elements $h^{\pm} \in \Gamma$ such that $h^{\pm} \overline{v}^{\pm} \in U^{\pm}$, respectively. The element $q=h^{+} g^n (h^{-})^{-1}$ will be very proximal with attracting and repelling points $\overline{v}_{q}^{\pm} \in U^{\pm}$. This proves (\ref{itm:dense_pairs}). If, $q,g$ are two very proximal elements with different attracting and repelling points, then $\{q^{n}\overline{v}_{g}^{+} \ | \ n \in \Z\}$ is an infinite set of points contained in $L(\Gamma)$. Thus $L(\Gamma)$ is an infinite compact minimal $\Gamma$-space, and hence perfect.  This concludes (\ref{itm:limit}). Finally (\ref{itm:ess_free}) follows from the fact that $|\Fix(\gamma)| < 3$ for every $e \ne \gamma \in \Gamma$. 
 \end{proof}

\begin{lemma}  \label{lem:hyp_prec}
Any very proximal element $\gamma \in \Gamma\le  \SL_2(k)$ admits precise attracting and repelling neighborhoods $\Omega_{\gamma}^{\pm} \subset L(\Gamma)$, in its action on $L(\Gamma)$. 
\end{lemma}
\begin{proof}
Given attracting and repelling neighborhoods $\Omega_1^{\pm} \subset L(\Gamma)$, we replace them by precise neighborhoods by setting $\Omega^{-}:=\overline{\mathring{\Omega_1^{-}}}$ and $\Omega^{+}:= \overline{\gamma(L(\Gamma) \setminus \Omega^{-})}$. It is clear that $\Omega^{\pm}$ thus defined are closed that the sets $\{\gamma^k O \ | \ k \in \Z\}$ are pairwise disjoint where $O = L(\Gamma) \setminus (\Omega^{-} \bigcup \Omega^{+} )$. Given any $x \in L(\Gamma) \setminus \{\overline{v}^{+}_{\gamma}, \overline{v}^{-}_{\gamma}\}$ let $k \in \Z$ be the largest number such that $\gamma^k x \not \in \Omega^{+}.$ 
Then $\gamma^k x$ is neither in $\Omega^+$ nor in $\mathring\Omega^-$, since $\gamma^{k+1}x\in\Omega^+$. It follows that $\gamma^kx\in O$.
\end{proof}

In this rank one setting we can obtain a more precise version of the ping-pong Lemma.  \ref{lem:ping-pong} 
\begin{lemma} \label{lem:precise_ping_pong} (Precise ping-pong lemma)  Let $S = \{\gamma_1, \gamma_2,\ldots,\gamma_N\} \subset \SL_2(k)$ be a collection of $N \ge 2$ very proximal elements. Set $\Delta = \langle S \rangle$. Suppose that $\{\Omega^{\pm}_i \subset \PP\}_{i = 1}^{N}$ are pairwise disjoint, {\it{precise}} attracting and repelling neighborhoods for the $\gamma_i$ action on $\PP$. Set $\Omega_i = \Omega^{+}_i \bigcup \Omega^{-}_i$ and $\Omega = \bigcup_{i=1}^{N} \Omega_i$ and assume that $O = \PP \setminus \Omega$ is nonempty. Then:
\begin{enumerate}
\item \label{itm:ind} $S$ freely generates a free group $\Delta$. 
\item \label{itm:limit2} There is a unique minimal closed $\Delta$-invariant subset $L=L(\Delta) \subset \PP$. 
\item \label{itm:map} There is a $\Delta$-equivariant, homeomorphism $\ell: \partial \Delta \arrow L$. 
\item \label{itm:nd} $\Delta O$ is open dense and $L$ is nowhere dense in $\PP$. 
\item \label{itm:fd} $O$ is a fundamental domain for the action $\Delta \curvearrowright \PP \setminus L.$  
\end{enumerate}
\end{lemma}
\begin{proof}
Statement (\ref{itm:ind}) is the standard ping-pong lemma, as proved for example in \cite[Proposition 3.4]{G:conv_non_ig}. 

Let us denote by $\overline{v_i}^{\pm} \in \PP$ the attracting and repelling points for $\gamma_i$. We know that $\lim_{n \arrow \infty} \gamma_1^n(x) = \overline{v}_{1}^{+}, \forall x \ne \overline{v}_1^{-}$ and $\lim_{n \arrow \infty} \gamma_1^n \gamma_2 (\overline{v}_{1}^{-}) = \overline{v}_{1}^{+}$. Set $L:=\overline{\Delta (\overline{v}_{1}^{+})}$. Then $L$ is contained in every closed $\Delta$-invariant set proving (\ref{itm:limit2}). Note that $L$ contains every attracting or repelling point for every $e \ne \delta \in \Delta$. 

We identify $\partial \Delta$ (the geometric boundary of the free group $\Delta$), with infinite reduced words in $S \sqcup S^{-1}$. If $\xi \in \partial \Delta$ let $\xi(k) \in \Delta$ denote its $k$-prefix. Now define recursively a map $\ell: \Delta \setminus \{e\} \arrow \Cl(\PP)$ of $\Delta$ into the space of closed subsets of $X$, by setting $\ell(s^{\epsilon}) = \Omega_s^{\epsilon}$ for $s \in S, \epsilon \in \{\pm 1\}$ and $\ell(s^{\epsilon} w) = s^{\epsilon} \ell(w)$ whenever $w \in \Delta$ is represented by a reduced word that does not start with $s^{-\epsilon}$. The ping-pong dynamics yields two properties that are easy to verify:
(i) $\ell(w) \subset \ell(v)$ whenever the reduced word representing $v$ is a prefix of that representing $w$, (ii) $\gamma \ell(w) \subset \ell(\gamma w)$ whenever $\gamma, w \in \Delta$ and $w$ is represented by a long enough word in the generators. 

Now we extend this definition to $\partial \Delta$ by setting $\ell(\xi) = \bigcap_{k \in \N} \ell(\xi(k))$ for every $\xi \in \partial \Delta$. $\ell(\xi)$ is nonempty by the finite intersection property. Using the metric contraction properties of the very proximal elements $\gamma^{\pm}_i$ given in Lemma \ref{lem:equiv_contraction}, one verifies that $\ell(\xi(k)) \in \PP$ is a single point. Thus $\ell$ defines a point map, which by abuse of notation we will still denote by $\ell: \partial \Delta \arrow \PP$. By property (i) above this map is injective, by property (ii) it is $\Delta$-invariant. Assume that $\xi_n \rightarrow \xi$ is $\partial \Delta$. By the definition of the standard topology on $\partial \Delta$ this just means that for every $m \in \N$ the m-prefixes $\{ \xi_n(m) \ | \ n \in \Z\}$ eventually stabilize, and are equal to $\xi(m)$. Consequently $\ell(\xi_n) \in \ell(\xi(m))$ for every $n$ large enough, which immediately implies continuity. Because both spaces are compact and metric $\ell$ is a homeomorphism onto its image. But the image is a minimal $\Delta$ set, hence equal to $L(\Delta)$ by (\ref{itm:limit2}). This concludes the proof of (\ref{itm:map}). 

Set $O_i = \PP \setminus \Omega_i$ and note that $O = \bigcap_{i =1}^{N} O_i$. By our assumption $O_i$ is a fundamental domain for the action of $\langle \gamma_i \rangle$ on $\PP \setminus \{\overline{v}_{i}^{\pm}\}$. Note that, since the sets $\Omega_i$ are disjoint $\overline{O} = \bigcap_{i=1}^{N} \overline{O_i}$. 
It is clear from the ping-pong dynamics that the $\Delta$ translates of this set are disjoint. Thus to demonstrate (\ref{itm:fd}) we take $x \in \PP \setminus L(\Delta)$ and produce some element $\delta \in \Delta$ such that $\delta x \in \overline{O}$. We achieve this by an inductive procedure setting $x_0:=x$ and defining a sequence of points $x_0,x_1,\ldots \subset \Delta x_0$, stoping on the first time that we hit $\overline{O}$. Thus if $x_{m} \in \overline{O}$ we are finished. If not there is a unique index $1 \le i_m \le N$ such that $x_m \not \in \overline{O}_{i_m}$. But $x_m \not \in \{\overline{v}_{i_m}^{\pm}\} \subset L(\Delta)$ so we can find some $n_m \in \Z$ so that $\gamma_{i_m}^{-n_m} x_m \in \overline{O}_{i_m}$. This inductive procedure must terminate after finitely many steps. Indeed it is easy to verify by induction that $\gamma_0^{n_0} \gamma_1^{n_1} \ldots \ldots \gamma_m^{n_m}$ is a reduced word and that $x \in \ell(\gamma_0^{n_0} \gamma_1^{n_1} \ldots \ldots \gamma_m^{n_m})$. If the procedure never terminates we will obtain an infinite reduced word $\xi = \gamma_0^{n_0} \gamma_1^{n_1} \ldots \ldots \gamma_m^{n_m} \ldots \in \partial \Delta$ with $x \in \ell(\xi)$, contradicting our assumption that $x \not \in L(\Delta)$.   

Finally it follows directly from the above that $\Delta O$ is a dense open subset contained in $X \setminus L(\Delta)$. Proving (\ref{itm:nd}). 
\end{proof}

\subsection{Possible partial permutations}
A {\it{possible partial permutation}} is a triplet of the form $\phi=(m,\alpha,\beta)$ with $m \in \N$, $\alpha = (a_1,a_2,\ldots, a_m)$, $\beta = (b_1,b_2,\ldots,b_m) \in \Gamma^m$. A possible partial permutation is called {\it{special}} if $a_1=b_1=e$. We will use the notation $\phi = (m(\phi),\alpha(\phi),\beta(\phi))$ to emphasize the data which is associated with a given possible partial permutation $\phi$. Denote by $\ppp = \ppp(\Gamma)$ the set of all possible partial permutations of $\Gamma$, and by $\ppp^{0}$ the collection of special ones. A given $\phi \in \ppp$ is said to be {\it{legitimate}} modulo a subgroup $\Delta < \Gamma$ if both $\alpha,\beta$ give rise to $m$ distinct elements 
$$
 \alpha \Delta = \{\alpha_{i} \Delta \ | \ 1 \le i \le m\},~\beta \Delta = \{\beta_{i} \Delta \ | \ 1 \le i \le m\} \subset \Gamma/\Delta.
$$ 
Such a legitimate $\phi \in \ppp$ defines a partial map on $\Gamma/\Delta$, which we denote by the same letter $\phi: \alpha \Delta \arrow \beta \Delta,$ given by $\phi(\alpha_{i} \Delta) = \beta_{i} \Delta, \ 1 \le i \le m$. Finally if there exists an element $\gamma \in \Gamma$ such that the partial map $\phi$ is the restriction of the quasiregular action $\gamma: \Gamma/\Delta \arrow \Gamma/\Delta$ we will say that $\phi$ is {\it{realized by $\gamma$}}. With this terminology in place note the following characterization of highly transitive actions in terms of the properties of a stabilizer of a point. 
\begin{lemma}\label{lem:PPP}
Let $\Delta < \Gamma$ be a subgroup and $\psi: \Gamma \arrow \Sym(\Gamma/\Delta)$ the corresponding transitive action. Then the following conditions are equivalent
\begin{itemize}
\item $\psi$ is highly transitive (i.e. has a dense image),
\item every possible partial permutation $\phi \in \ppp(\Gamma)$ which is legitimate modulo $\Delta$ is realized by some $\gamma \in \Gamma$,  
\item every special possible partial permutation $\phi \in \ppp^{0}(\Gamma)$ which is legitimate modulo $\Delta$ is realized by some $\delta \in \Delta$. 
\end{itemize}
\end{lemma}
\begin{definition}
A subgroup $\Delta < \Gamma$ satisfying the equivalent conditions of Lemma \ref{lem:PPP} will be called {\it{co-highly transitive}} or {\it{co-ht}} for short. 
\end{definition}
\subsection{High transitivity proofs}


\begin{proof}[Proof of Theorem \ref{thm:rank_one_ht}]
The equivalence of (\ref{r1_ast}), (\ref{r1_zd}) and (\ref{r1:nvs}) follows directly from Lemma \ref{lem:linear_ast} combined with the fact that every connected proper algebraic subgroup of $\SL_2$ is solvable. That (\ref{r1_ht}) implies (\ref{r1_ast}) follows from Proposition \ref{prop:nec} combined with the fact that primitive groups of affine and of diagonal type are never highly transitive. Indeed it is proven in that proposition that these groups are semidirect products of the from $\Delta \ltimes M$ and that their unique primitive action is the standard affine action $\Gamma \action M$. If this action were to be highly transitive it would follow that the conjugation action of $\Delta$ on $M \setminus \{e\}$ is highly transitive, which is absurd for any group. Thus we remain with our main task for this section: to construct a co-ht core free subgroup $\Delta$ in $\Gamma$ whenever $\Gamma < \SL_2(k)$ is of almost simple type. We construct such a subgroup that is infinitely generated Schottky, coming from a ping-pong game on $Y = L(\Gamma) \subset \PP(k^2)$. The limit set of $\Gamma$, given by Lemma \ref{lem:limit}(\ref{itm:limit}). 

Let $\ppp^0 = \{\phi_2,\phi_4,\phi_6, \ldots \}$ be an enumeration of all special possible partial permutations of $\Gamma$. Let $\{ \gamma_1 \langle \langle n_1 \rangle \rangle,  \gamma_3 \langle \langle n_3 \rangle \rangle,  \gamma_5 \langle \langle n_5 \rangle \rangle, \ldots \}$ be an enumeration of all cosets of all the normal subgroups that are generated by a nontrivial conjugacy class. We construct, by induction on $\ell$, an increasing sequence of precise, spacious Schottky groups $\Delta_{\ell}  = \langle \delta_1,\delta_2,\ldots, \delta_{N_{\ell}}\rangle$ subject to the following properties:
\begin{itemize}
\item $(\delta_1,\delta_2,\ldots, \delta_{N_{\ell}})$ is a precise spacious ping-pong tuple on $L(\Gamma)$. 
\item For $\ell$ even, if $\phi_{\ell}$ is legitimate modulo $\Delta_{\ell-1}$ then it is realized  modulo $\Delta_{\ell}$ (by some element of $\Delta_{\ell}$).  
\item For $\ell$ odd, $\Delta_{\ell} \bigcap \gamma_{\ell} \langle \langle n_{\ell} \rangle \rangle \ne \emptyset$. 
\end{itemize}
Setting $\Delta = \bigcup_{\ell} \Delta_{\ell} = \langle \delta_1, \delta_2, \ldots \rangle$, the condition imposed on the odd steps guarantees that $\Delta$ will be prodense, and in particular core free. The even steps ensure that it is co highly transitive, by virtue of Lemma \ref{lem:PPP}. Hence the coset action $\Gamma \action \Gamma/\Delta$ will be faithful and highly transitive. 

The odd steps of the induction were already treated, in greater generality, in the proof of Theorem \ref{thm:ast1}. We will not repeat the argument here and turn directly to the even steps. Fix an even $\ell$. For the sake of a better legibility, we will often omit $\ell$ from the notation. For example we will denote $\phi = \phi_{\ell} = (m,\alpha,\beta)$. If $\phi = \phi_{\ell}$ is illegitimate modulo $\Delta_{\ell-1}$ then it would definitely be illegitimate modulo any larger subgroup. In this case we just declare $\Delta_{\ell} = \Delta_{\ell - 1}$. From here on we suppose that $\phi_{\ell}$ is legitimate modulo $\Delta_{\ell-1}$. 

By our induction assumption $\Delta_{\ell-1}$ is generated by a precise ping-pong tuple $\Delta_{\ell-1} = \langle \delta_1,\delta_2, \ldots, \delta_{N_{\ell-1}} \rangle$. We will find $\gamma \in \Gamma$ such that 
$
 \left \{ \delta_1,\delta_2  , \ldots, \delta_N, \gamma, b_2^{-1} \gamma a_2, \ldots, b_m^{-1} \gamma a_m \right \}
$
still constitutes a precise ping-pong tuple. Setting $\Delta_{\ell}$ to be the group generated by these elements, after renaming them appropriately\footnote{ That is, setting $\delta_{N_{\ell-1}+1} = \gamma, \delta_{N_{\ell-1}+2}=b_2^{-1} \gamma a_2, \ldots, \delta_{N_{\ell-1}+m}=b_m^{-1} \gamma a_m,$ and setting $N_{\ell} = N_{\ell-1} + m$.}, one  verifies that $\phi_n$ is now realized by $\gamma \in \Delta_{\ell}$ modulo $\Delta_{\ell}$. 

Let $\{\Omega_i^{\pm}\}_{i = 1 \ldots N_{\ell -1}}$ be precise attracting and repelling points for the generators of $\Delta_{\ell-1}$ and $O = L(\Gamma) \setminus \bigcup_{i =1}^{N_{\ell-1}} (\Omega_i^{-} \cup \Omega_i^{+})$ the fundamental domain, given by Lemma \ref{lem:precise_ping_pong}, for the action $\Delta_{\ell-1} \action L(\Gamma) \setminus L(\Delta_{\ell-1})$. Since, by assumption, $\Delta_{\ell-1}$ is spacious Schottky inside $\Gamma$ the set $L(\Gamma) \setminus L(\Delta_{\ell -1})$ is nonempty. Hence by Item (\ref{itm:nd}) in that same Lemma $O \subset L(\Gamma)$ is relatively open and dense. 

Consider the sets 
$$
 R = \bigcap_{i=1}^m a_i^{-1} \Delta O, A = \bigcap_{i=1}^m b_i^{-1} \Delta O \subset L(\Gamma).
$$ 
By the Baire category theorem both sets are open and dense in $L(\Gamma)$. In particular we can find two points $a \in A, r \in R$. Let 
$$
 \{\theta_j \ | \ 1 \le j \le m\}, \{\eta_j \ | \ 1 \le j \le m\}
$$ 
be the unique elements of $\Delta_{\ell-1}$ satisfying $\theta_j^{-1} a_j^{-1} r \in O {\text{ and }} \eta_j^{-1} b_j^{-1} a \in O, \ \forall 1 \le j \le m$. Thus the sets
\begin{eqnarray*}
R_1  & := & \{x \in R \ | \ \theta_j^{-1} a_j^{-1} x \in O, \quad \forall 1 \le j \le m\} {\text{ and }} \\
A_1 & := & \{y \in A \ | \ \eta_j^{-1} b_j^{-1} y \in O, \quad \forall 1 \le j \le m \},
\end{eqnarray*}
are open and nonempty. 

By Lemma \ref{lem:limit} (\ref{itm:dense_pairs}) we can find a very proximal element $\gamma \in \Gamma$ with repelling and attracting points $(\overline{v}^{-},\overline{v}^{+})$ subject to the following open conditions:
\begin{enumerate}
\item \label{itm:dd} $\overline{v}^{-} \in R_1, \overline{v}^{+} \in A_1$,
\item \label{itm:supp} $\overline{v}^{-} \in \Supp(a_k \theta_k \theta_j^{-1} a_j^{-1})$ and $\overline{v}^{+} \in \Supp(b_k \eta_k \eta_j^{-1} b_j^{-1}), \ \ \forall 1 \le j \ne k \le m$, 
\item $b_k \eta_k \theta_j^{-1} a_j^{-1} \overline{v}^{-} \ne \overline{v}^{+}, \qquad \forall 1 \le j,k \le m$.
\end{enumerate}
For the second conditions we used Lemma \ref{lem:limit}(\ref{itm:ess_free}) together with our assumption that $\phi$ is legitimate modulo $\Delta_{\ell-1}$ to ensure that $a_k \theta_k \theta_j^{-1}a_j^{-1} \ne e$ and $b_k \eta_k \eta_j^{-1} b_j^{-1} \ne e$. 

The above choices guarantee that 
$$
 \{r_j:= \theta_j^{-1} a_j^{-1} \gamma^{-}, a_j:=\eta_j^{-1} b_j^{-1} \gamma^{+} \ | \ 1 \le j \le m \}
$$ are $2m$ distinct points inside $O$. We propose a precise ping-pong generating set for $\Delta_{\ell}$ of the form
$$
 \left \{ \delta_1,\delta_2  , \ldots, \delta_N, \gamma^k, (b_2 \eta_2)^{-1} \gamma^k(a_2 \theta_2), \ldots,  (b_m \eta_m)^{-1} \gamma^k(a_m \theta_m) \right \}.
$$
If we can adjust the parameter $k$ so that these are indeed a precise ping-pong tuple, all the desired properties hold. In particular $\gamma^k \in \Gamma$ would be the element realizing $\phi$. 

But if $\Omega^{\pm}$ are repelling and attracting neighborhoods for $\gamma^{k}$ then 
\begin{equation} \label{eqn:neb}
r_j \in \theta_j^{-1} a_j^{-1} \Omega^{-}, a_j \in \eta_j^{-1} b_j^{-1} \Omega^{+},
\end{equation}
serve as repelling and attracting neighborhoods for $(b_j \eta_j)^{-1} \gamma^k(a_j \theta_j)$. By setting $k$ large enough we can make the neighborhoods $\Omega^{\pm}$ arbitrarily small. Using Lemma \ref{lem:hyp_prec} we can impose the condition that these sets are precise, pairwise disjoint and contained in $O$. This completes the proof of the theorem. 
 \end{proof} 
 
 \section{Groups which are not linear-like} \label{sec:nonlin}
The proof of Theorem \ref{thm:MS} was based on the construction of a profinitely dense subgroup and then passing to a maximal subgroup containing it. We have seen by now many variants on this idea. One common feature of all constructions described so far was that the profinitely dense subgroup was free, and constructed by various ping-pong games a-la Tits. Consequently all primitive groups that we have encountered so far are {\it{large}} in the sense that they contain a nonabelian free subgroup. One notable exception were some of the primitive groups of affine and diagonal type that we encountered. 

This is, in fact a feature of the methods we used so far. There are many primitive groups of almost simple type that do not contain free subgroups. In fact there are even linear examples. Take the group $\PSL_n(K)$ where $K$ is any countable locally finite field. For instance one could take $K = \overline{F_p}$ to be the algebraic closure of $F_p$. These are all primitive groups of almost simple type, even though they are locally finite. In the case $n=2$ the group $\PSL_n(K)$ is even 3-transitive, by virtue of its standard action on the projective line $\PP(K)$. The goal of this section is to give a short survey, devoid of proofs, on some other results which are of different nature than the ones considered in previous sections. 

\subsection{Some groups of subexponential growth} It is easy to verify that the Grigorchuk group $G$ is of almost simple type. Though it is not stably so in the sense that many of its finite index subgroups fail to be of almost simple type. In fact $G_n < G$, the stabilizer in $G$ of the $n^{th}$ level of the tree splits as the direct product of $2^n$ subgroups. 

It was asked by Grigorchuk whether all maximal subgroups of the first Grigorchuk group are of finite index. This question was answered positively by Pervova:
\begin{theorem}[\cite{per:max}] Every maximal subgroup of the Grigorchuk group is of finite index. 
\end{theorem}
Thus we have an example of a residually finite, finitely generated group of almost simple type admitting no maximal subgroup of infinite index. The results of Pervova were later generalized in \cite{AKT:max} to encompass a much larger family of groups. However there are groups of subexponential growth that admit infinite index maximal subgroups as shown by the following remarkable theorem of Francoeur and Garrido. The theorem deals with a family of groups that they coin {\it{\v{S}uni\'{c} groups}}, as they were defined by Zoran \v{S}uni\'{c} in \cite{Sun:family}. We refer the readers to one of these articles for the precise definition of these groups, remarking here only that it is a family of finitely generated groups, of sub-exponential growth, acting on the binary rooted tree. 
\begin{theorem}[\cite{FG:max}] Let $G$ be a nontorsion  \v{S}uni\'{c}  group that is not isomorphic to the infinite dihedral group, acting on the rooted binary tree. Then $G$ admits countably many maximal subgroups, out of which only finitely many are of finite index. 
\end{theorem}
Other examples of groups admitting countably many maximal subgroups of infinite index include Tarski monsters (that have only countably many subgroups to begin with) as well as some examples of affine type constructed by Hall in \cite{Hall:fin}.

\subsection{Thompson's group $F$} The Abelianization $F/F'$ of the Thompson group is isomorphic to $\Z^2$. Any nontrivial normal subgroup of $F$ contains the commutator, thus $F$ is clearly of almost simple type. A subgroup $\Delta < F$ is profinitely dense if and only if it maps onto the abelianization so it is not surprising that $F$ contains many infinite index maximal subgroups. 

In \cite{sv:Fgraph1,sv:Fgraph2} Savchuk proves that all the orbits of the natural action of $F$ on the interval $(0,1)$ are primitive. Otherwise put $F_{v} < F$ is a maximal subgroup of infinite index for every $v \in (0,1)$. A notable fact is that many of these maximal subgroups are even finitely generated. Savchuk shows in particular that $F_v$ is finitely generated whenever $v \in \Z[1/2]$ and that it fails to be finitely generated whenever $v$ is irrational. 

All of the facts mentioned above are not difficult to verify. Svachuk asked whether Thompson's group $F$ contains maximal, infinite index subgroups that fail to fix a point. In a beautiful paper, Golan and Sapir \cite{GS:max} answer this question positively by constructing many very interesting examples of maximal subgroups. In particular they prove the following:
\begin{theorem}[\cite{GS:max}] The Jones group $\vec{F} < F$, constructed in \cite{J:Jgroup} is maximal of infinite index.
\end{theorem}
This is the same group constructed by Vaughn Jones for establishing connections between link theory and Thompson's group $F$. It was shown by Brown in \cite{BR:Jgp} that the group $\vec{F}$ is isomorphic to the {\it{triadic}} version of the group $F$ itself. Namely the group of all orientation preserving homeomorphisms of the interval $[0,1]$ which are piecewise linear with slopes that are powers of three and finitely many non-differentiability points, all of which are contained in $\Z[1/3]$. 

In addition to this one example, which is of particular interest due to the fact that this group is finitely generated, and of independent interest; Golan and Sapir provide general methods of constructing a large variety of infinite index maximal subgroups in $F$.

\subsection{Locally finite simple groups.}
By a famous theorem of Schur, every periodic linear group is locally finite. Thus in view of  Theorem \ref{GMS} and the discussion following it, amongst the countable linear groups,  locally finite groups of almost simple type are the only ones for which Question \ref{q:main} is still open. Within this family, locally finite simple groups constitute an interesting special case. 

In the emerging theory of locally finite simple groups (see for example the book \cite{Book:FLF}), linear groups also play a special role. In fact it is quite customary in this subject to sort the locally finite simple groups into four families of increasing complexity: (i) finite, (ii) linear (iii) finitary linear and (iv) the general case. Where:
 
\begin{definition}
A group $\Gamma$ is called {\it{finitary linear}} if there exits a vector space $V$ over a field $F$ and an embedding $\Gamma < \GL(V)$ with the property that $\Im(\gamma - I)$ is a finite dimensional subspace of $V$ for every $\gamma \in \Gamma$. 
\end{definition}

In view of all this we find the following theorem of Meierfrankenfeld extremely interesting. His theorem gives a very strong solution to Question \ref{q:main} exactly in the most complicated class of locally finite simple groups:
\begin{theorem}[\cite{MF:LFsimple}, Theorem B]
Let $\Gamma$ be a locally finite simple group that is not finitary linear. Then every finite subgroup $D < \Gamma$ is contained in some proper maximal subgroup.
\end{theorem}
Note that since $\Gamma$ here is simple every maximal subgroup is automatically core free. 

\subsection{Non essentially free homeomorphism groups of the circle and of $\partial T$} Here we wish to highlight a very recent theorem of Le Boudec and Matte Bon. We view this theorem as one of the only nontrivial obstructions currently known to a group being highly transitive. 
\begin{theorem}[\cite{BMB:triple}] 
Let $\Gamma < \Homeo(S^1)$. Assume that the action of $\Gamma$ on $S^1$ is proximal, minimal and {\it{not}} topologically free. Assume that distinct points in $S^1$ have distinct stabilizers. Then every faithful, 3-transitive action of $G$ on a set is conjugate to its given action on one of the orbits within $S^1$.
\end{theorem}
In the same paper the authors prove a very similar theorem for group actions on a regular locally finite tree. This enables the authors to bound the transitivity degree of many groups. For example for the natural index two extension $T^{\pm}$ of Thompson's group $T$ (generated by $T$ and a reflection of the circle) they show that it admits a $3$-transitive action on a set (the set of dyadic points on the circle) but it does not admit a $4$-transitive action. 

\bibliographystyle{alpha}
\bibliography{./yair}

\noindent {\sc Yair Glasner.} Department of Mathematics.
Ben-Gurion University of the Negev.
P.O.B. 653,
Be'er Sheva 84105,
Israel.
{\tt yairgl\@@math.bgu.ac.il}

\bigskip

\noindent {\sc Tsachik Gelander.} Department of Mathematics.
Weizmann Institute of Science.
Rehovot 76100,
Israel.
{\tt tsachik.gelander\@@weizmann.ac.il}

\bigskip

\noindent {\sc Gregory So\u\i fer.} Department of Mathematics.
Bar-Ilan university.
Ramat-Gan, 5290002 Israel.
{\tt So\u\i fer\@@math.biu.ac.il}

\end{document}